\documentclass[journal]{IEEEtran}

\usepackage{xcolor}
\usepackage{graphics} 
\usepackage{epsfig} 
\usepackage{amsmath} 

\usepackage{amssymb,amsthm}  
\usepackage{amscd}
\usepackage[noadjust]{cite}
\usepackage{float}
\usepackage{times}
\usepackage{enumitem}
\newtheorem{theorem}{Theorem}
\newtheorem{remark}{Remark}
\newtheorem{lemma}{Lemma}

\newtheorem{Remark}{Remark}

\newtheorem{corollary}{Corollary}
\newtheorem{definition}{Definition}

\newcommand{\F}{\mathbb{F}}
\newcommand{\cS}{\mathcal{S}}

\ifCLASSINFOpdf
\else
\fi

\begin{document}
%
\title{Asymptotic Behavior of Conjunctive Boolean Networks Over Weakly Connected Digraphs}
%
%
%

\author{
	Xudong~Chen,~\IEEEmembership{Member,~IEEE,}
        Zuguang~Gao,~
         and~Tamer~Ba\c{s}ar,~\IEEEmembership{Life~Fellow,~IEEE}
\thanks{This research was supported in part by the Office of Naval Research (ONR) MURI Grant N 00014-16-1-2710. This work is based upon a preliminary version ``Asymptotic behavior of a reduced conjunctive Boolean network'' presented at the 56th IEEE Conference on Decision and Control (CDC), 2017.}
\thanks{X. Chen is with Department of Electrical, Computer, and Energy Engineering, University of Colorado Boulder, Boulder, CO 80309, email: xudong.chen@colorado.edu.}
\thanks{Z. Gao is with Booth School of Business, The University of Chicago, Chicago, IL 60637, email: zuguang.gao@chicagobooth.edu}
\thanks{T. Ba\c{s}ar is with Coordinated Science Laboratory, University of Illinois at Urbana-Champaign, Urbana, IL 61801, email: basar1@illinois.edu }
}

\maketitle

\begin{abstract}
  A conjunctive Boolean network (CBN) is a finite state dynamical system, whose variables take values from a binary set, and the value update rule for each variable is a Boolean function consisting only of  logic AND operations.  
We investigate the asymptotic behavior of CBNs by computing their periodic orbits.  
When the underlying digraph is strongly connected, the periodic orbits of the associated CBN has been completely understood, one-to-one corresponding to binary necklaces of a certain length given by the loop number of the graph.    
We characterize in the paper the periodic orbits of CBNs over an arbitrary {\em weakly connected} digraphs. We establish, among other things, a new method to investigate their asymptotic behavior. Specifically, we introduce a graphical-approach, termed {\em system reduction}, which turns the underlying digraph into a special weakly connected digraph whose strongly connected components are all cycles. We show that the reduced system uniquely determines the asymptotic behavior of the original system. Moreover, we provide a constructive method for computing the periodic orbit of the reduced system, which the system will enter for a given but arbitrary initial condition.     
\end{abstract}


%
\IEEEpeerreviewmaketitle

\section{Introduction}
%
%
%
%
%
%
%

A Boolean network is a 
finite state dynamical system whose variables are of Boolean type, labeled as ``$1$'' and ``$0$''. The value update rule for each variable is a Boolean function, depending only on a selected subset of the variables. Boolean networks have a wide range of applications in biochemistry, molecular biology, genetics, genomics and neuroscience, to name just a few, and can serve as efficient models for biological regulatory systems, such as neural network~\cite{hopfield1982neural,hopfield1984neurons,mcculloch1943logical} or gene regulatory networks~\cite{kauffman1969metabolic}.
This line of research began with Boolean network representations of molecular networks \cite{kauffman1969homeostasis}, and was later generalized to the so-called logical models \cite{thomas1990biological}. Since then there have been studies of various classes of Boolean networks which are particularly suited to the logical expression of gene regulation \cite{raeymaekers2002dynamics}.

A special class of Boolean networks, of particular interest to us here, is the so-called \emph{conjunctive Boolean networks} (CBNs).   
Roughly speaking, a CBN is such that the Boolean functions for the variables are comprised only of the logic ``AND'' operations (a precise definition is given in Subsection~\ref{ssec:CBN}). 
Even though the update rule of a CBN is relatively simple, it has several relevant properties, which are critical in modeling complex network systems. For example, a CBN is {\em monotonic}, i.e., the output value of a Boolean function for each variable is non-decreasing if the number of $1$'s in the inputs increases. Evidence has been provided in~\cite{sontag2008effect} that biochemical networks are ``close to monotone''. We also note that each Boolean function in a CBN is a {\em canalyzing function}~\cite{jarrah2007nested}, meaning that if an input of the function holds a certain value, called the {\em canalyzing value}, then the output value of the function is uniquely determined regardless of the other values of the inputs. 
For example, the canalyzing value is~$0$ for a Boolean function in a CBN.     
Boolean networks with canalyzing functions are used to model genetic networks~\cite{harris2002model,kauffman2003random}.
 We further note that the class of CBNs is ``universal'' in the sense that every general Boolean network can be associated with a CBN whose dynamic behavior uniquely determines the dynamic behavior of the original Boolean network. We make it precise in Sec.~\ref{ssec:universal}.
For the above reasons, CBNs have drawn special attention most recently. The stability structure of the periodic orbits in a strongly connected CBN has been investigated in~\cite{stabilityfull}. Controllability and Observability of CBNs have been addressed in~\cite{gaocontrollability,orbitcontrol,statecontrol,weiss2017minimal,weiss2017polinomial}.

We characterize in the paper the asymptotic behavior of a CBN over a weakly connected digraph.  
Since a CBN is a finite dynamical system, for any initial condition the trajectory generated by the system enters a periodic orbit (also known as a limit cycle) in finite time steps~\cite{colon2005boolean}. The problem of counting and characterizing periodic orbits of Boolean networks has been studied from different aspects via different approaches. We first refer the reader to~\cite{cheng2010linear,cheng2010analysis} for algebraic methods of using semi-tensor products by which one converts a Boolean function into an equivalent algebraic form. We next refer the reader to~\cite{zhao2013aggregation} for methods of computing periodic orbits using the corresponding state-transition graph. Note that if a Boolean network has $n$ variables, then its state-transition graph will have as many as $2^n$ vertices. To mitigate the computational complexity of finding all cycles in the state-transition graph, the authors of~\cite{zhao2013aggregation} proposed an approach of first decomposing a Boolean network into several subnetworks, next finding all input-state limit cycles of those subnetworks (using their state-transition graphs), and then composing the input-state limit cycles by checking certain compatibility conditions.

In this paper, we take a completely different graphical approach to tackle the problem. Instead of working on the state-transition graph  that corresponds to a CBN (or any subnetwork of the CBN), we appeal to the idea of system reduction and introduce a new but significantly simplified CBN, termed {\em reduced system}, which exhibits essentially the same asymptotic behavior as that of the original CBN.   
In particular, the simplification will be made such that the underlying digraph of the reduced system is a weakly connected digraph whose strongly connected components are all cycles of  relatively small sizes. The key fact that relates the dynamic behaviors of the reduced system and the original system  is given in  Theorem~\ref{pro:dynamics}, Subsection~\ref{ssec:thm1}. 
Specifically, we will show that the asymptotic behavior of the original CBN can be uniquely determined by the corresponding  reduced system. One is thus able to characterize the periodic orbits of the original system by working on the reduced system. 
Furthermore, after introducing the approach of system reduction and establishing Theorem~\ref{pro:dynamics}, 
we address the following question: Given an arbitrary initial state of the original CBN, which periodic orbit will the system eventually enter? We provide in Subsection~\ref{ssec:general} a complete solution to this question via the use of a reduced system. In particular, we will construct a map which assigns an arbitrary initial condition to one particular state in the periodic orbit which the system will enter. 

This paper expands on our previous work~\cite{chen2017asymptotic} by providing several critical properties about the dynamics of (weakly connected) CBNs, a finer and more thorough analysis of their asymptotic behavior, and complete proofs of the lemmas and the main results that were left out of that conference paper. 
We also note that asymptotic behaviors of CBNs have been studied over strongly connected digraphs. 
For example, it has been shown in~\cite{jarrah2010dynamics,gao2016periodic} that if the underlying digraph is strongly connected, then a positive integer is the period of a certain periodic orbit if and only if it divides the lengths of all cycles in that corresponding digraph. We have further shown in~\cite{stabilityfull} that the set of periodic orbits can be identified with the set of binary necklaces, with the length being the greatest common divisor of all these cycles' lengths (a detailed overview of this fact will be given in Subsection~\ref{ssec:3sc}). However, such a one-to-one correspondence does not carry over to the case where the underlying graph is only weakly connected.

The remainder of the paper is organized as follows: 
In Section~\ref{pre}, we provide key definitions and notations for digraphs, binary necklace, and CBNs as well as the universality property of the class of CBNs. 
In section~\ref{reduction}, we introduce the notion of a reduced system as well as the associated induced dynamics. Furthermore, we show that the induced dynamics defined on the reduced system uniquely determines the asymptotic behavior of the original system. Thus, the analysis can be simplified by investigating only the asymptotic behavior of a reduced system. This is done in Subsection~\ref{ssec:elem}. Specifically, we introduce there a simple class of digraphs, termed {\em elementary digraphs}, and characterize the asymptotic behaviors of the CBNs defined on these digraphs. We then show in Subsections~\ref{ssec:simpleunion} and~\ref{ssec:general} that one can use the elementary digraphs as building blocks to characterize  the asymptotic behavior of a general reduced system. 
The paper ends with conclusions and outlooks in Section~\ref{end}. 

\section{Preliminaries}\label{pre}
We introduce here definitions and notations about digraphs, binary necklaces, and CBNs.  
\subsection{Digraphs and their strong component decompositions}
We introduce here some notations associated with a digraph. Let $G=(V,E)$ be a digraph, with $V$ the vertex set and $E$ the edge set. The cardinality of $V$, denoted by~$|V|$, is the number of vertices. 
We denote by $v_iv_j$ an edge from $v_i$ to $v_j$ in~$G$. We call $v_i$ an {\em in-neighbor} of  $v_j$ and $v_j$ an {\em out-neighbor} of $v_i$. 
We denote by $\mathcal{N}_{\rm in}(v_i)$ and
$\mathcal{N}_{{\rm out}}(v_i)$ the sets of in-neighbors and out-neighbors, respectively, of vertex $v_i$.  A {\em walk} in $G$ is a sequence of vertices $v_{i_0}v_{i_1}\cdots v_{i_m}$ in which each $v_{i_j}v_{i_{j+1}}$, for $j = 0,\ldots, m-1$, is an edge.  A walk is said to be a {\em path} if all the vertices in the walk are pairwise distinct. Further, a walk is said to be a {\em cycle}  if there is no repetition of vertices in the walk other than the repetition of the starting- and ending- vertex.

For a subset $V'\subseteq V$, we let $\mathcal{N}_{{\rm in}}(V'):= \cup_{v_i\in V'}\mathcal{N}_{{\rm in}}(v_i)$ and $\mathcal{N}_{{\rm out}}(V'):= \cup_{v_i\in V'}\mathcal{N}_{{\rm out}}(v_i)$. Further, for any positive integer $k$, we define ${\mathcal N}_{\rm in}^k(V')$ and ${\mathcal N}_{\rm out}^k(V')$ recursively, 
i.e., for $k > 1$, we let 
$$
\left\{
\begin{array}{l}
{\mathcal N}_{\rm in}^k(V') := \cup_{v_i \in {\mathcal N}_{\rm in}^{k-1}(V')}{\mathcal N}_{\rm in}(v_i), \\
{\mathcal N}^k_{\rm out}(V') := \cup_{v_i \in {\mathcal N}_{\rm out}^{k-1}(V')}{\mathcal N}_{\rm out}(v_i).
\end{array}
\right.
$$
Let $G'$ be a subgraph of $G$ and $v_i$ be a vertex of $G'$. We denote by ${\mathcal N}_{\rm in}(v_i; G')$ (resp. ${\mathcal N}_{\rm out}(v_i, G')$) the in-neighbors (resp. out-neighbors) of $v_i$, but within the graph $G'$. Similarly, we define ${\mathcal N}^k_{\rm in}(v_i; G')$ and ${\mathcal N}^k_{\rm out}(v_i; G')$.


A digraph is said to be {\em weakly connected} if the undirected graph, obtained by ignoring the orientations of the edges, is connected. 
A digraph is said to be \emph{strongly connected} if for any two distinct vertices~$v_i$ and~$v_j$ in the graph, there is a path from~$v_i$ to~$v_j$. A digraph comprised of only 
a single vertex, with/without the self-arc, is by default strongly connected. 
Let $G = (V, E)$ be a digraph, and $G_i = (V_i, E_i)$ and $G_j = (V_j, E_j)$ be two subgraphs of $G$. Then, $G_i$ and $G_j$ are said to be disjoint if $V_i \cap V_j = \varnothing$. Further, we say that a subgraph $G_i$ is induced by $V_i$ if {\color{black}$E_i = \{v_kv_l \mid v_kv_l\in E \text{ and }v_k, v_l \in V_i\}$. }  We  now introduce the following definition: 

\begin{definition}\label{def:SCD}
	Let $G = (V, E)$ be a weakly connected digraph. The subgraphs $G_i = (V_i, E_i)$, $1\le i\le q$, form the {\bf (coarsest) strong component decomposition} of $G$ if the following hold:
	\begin{enumerate}
		\item The subgraphs $G_i$'s are pairwise disjoint, and moreover, $V = \sqcup^q_{i = 1} V_i$. 
		\item Each $G_i$ is strongly connected, and moreover, there does not exist a strongly connected subgraph $G'_i = (V'_i, E'_i)$ such that $V_i \subsetneq V'_i$. 
	\end{enumerate}
\end{definition}

Note that the strong component decomposition (SCD) exists and is unique (see, for example~\cite{chen2017controllability}). We also note that the SCD induces a partial order on the collection of the subgraphs $G_i$'s. Specifically, for two distinct subgraphs $G_i$ and $G_j$, we say that $G_j$ is a {\em successor} of $G_i$ (and correspondingly, $G_i$ is a {\em predecessor} of $G_j$) if there exists a path from a vertex $v_i\in V_i$ to a vertex $v_j\in V_j$. For simplicity, we denote this relationship by $G_i \succ G_j$. 
Further, we say that $G_j$ is an {\em immediate successor} of $G_i$ if $G_i \succ G_j$, and moreover, there does not exist another $G_k$ such that $G_i \succ G_k \succ G_j$. A subgraph $G_i$ is a {\em maximal} element if there is no predecessor of $G_i$. 

Given a weakly connected graph $G$, we denote by $\cS(G)$ the collection of the subgraphs $G_i$'s of $G$ obtained by the SCD. Let $\cS^0(G)$ be a subset of $\cS(G)$ collecting the subgraphs that are maximal with respect to the partial order defined above. We then let $\cS^1(G)$ be the immediate successor of $\cS^0(G)$ defined as follows: Each $G_j \in \cS^1(G)$ is an immediate successor of some $G_i \in \cS^0(G)$; moreover, for any other $G_{i'}\in \cS^0(G)$,  either $G_{i'}$ and $G_j$ are not comparable or $G_j$ is an immediate successor of $G_{i'}$.  Similarly, we define $\cS^2(G)$ to be the immediate successor of $\cS^1(G)$. Since there are only finitely many subgraphs, there must exist an integer~$L$ such that $\cS^{L+1}(G) = \varnothing$. We may as well choose~$L$ to be the smallest integer for the relation above to hold. Then, each $\cS^l(G)$, $0\le l\le L $,  is nonempty, and moreover, we have that $\cS(G) = \sqcup^L_{l = 0} \cS^l(G)$, i.e., the subsets $\cS^l(G)$'s form a partition of~$\cS(G)$.


\subsection{Binary necklace}\label{binarynecklace}
A {\bf binary necklace} of length $p$ is an equivalence class of $p$-character strings over the binary set $\mathbb{F}_2=\{0,1\}$, taking all rotations as equivalent. For example, in the case $p = 4$, there are six different binary necklaces, as illustrated in Fig.~\ref{necklace}. 
The {\bf order} of a necklace is the cardinality of the corresponding equivalence class, and it is always a divisor of $p$.  We refer the reader to~\cite{moreau1872permutations} for the number of binary necklaces of a given length, and further, to~\cite{gilbert1961symmetry,ruskey1999efficient} for the number of binary necklaces with a given number of~$1$'s.

\begin{figure}[ht]
	\centering
	\includegraphics[height=40mm]{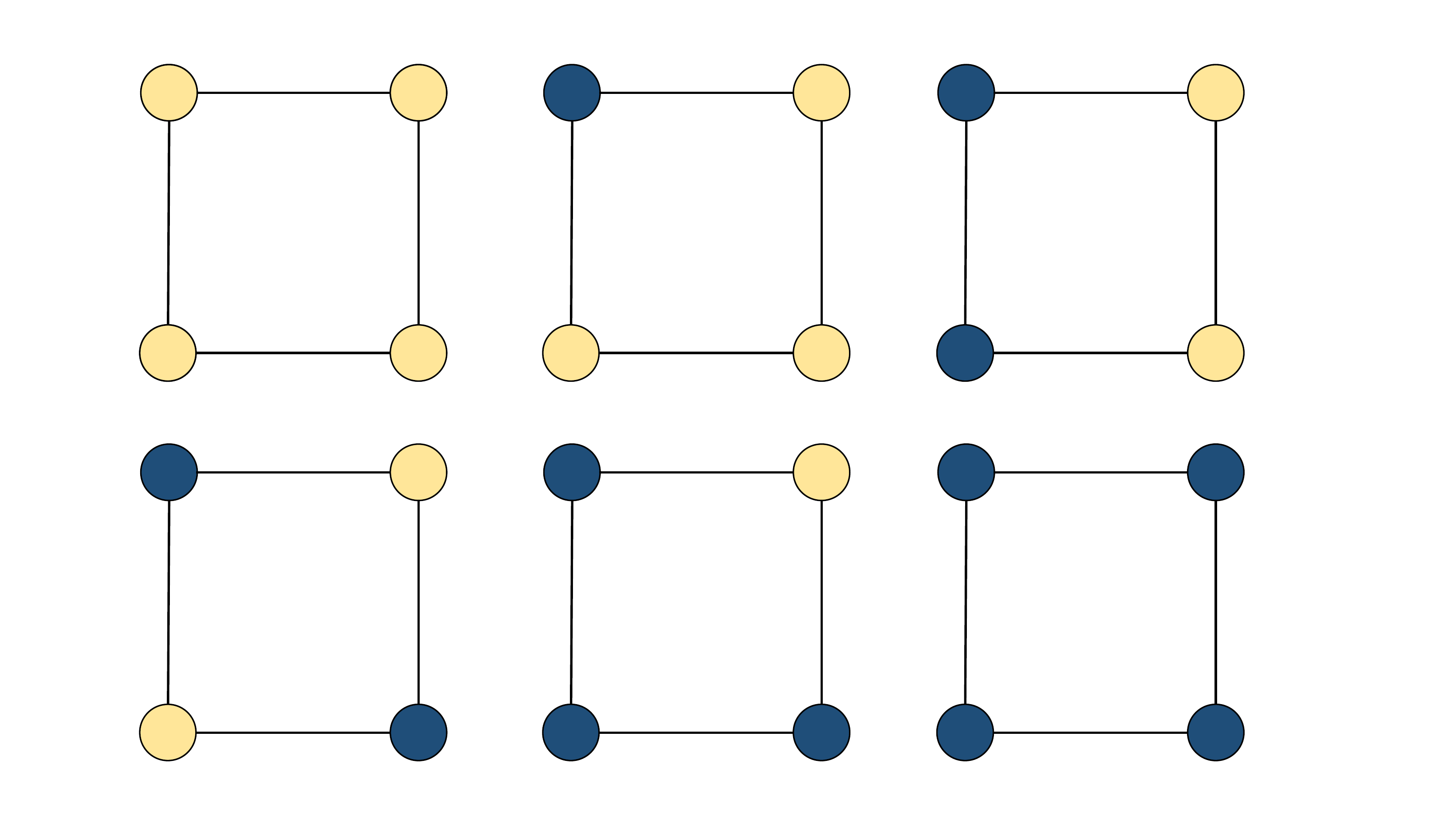}
	\caption{We illustrate here all binary necklaces of length~$4$. In the figure, if the bead is plotted in dark blue (resp. light yellow), then it holds value ``1'' (resp, ``0''). }
	\label{necklace}
\end{figure}

\subsection{Conjunctive Boolean network (CBN)}\label{ssec:CBN}

Let $\mathbb{F}_2=\{0,1\}$ be the finite field.
A function $f$ on $n$ variables is a {\em Boolean function} if it is of the form $f:\F^n_2 \to \F_2$. 
The {\em value update rule} associated with a Boolean network on $n$ variables $x_1(t),\ldots, x_n(t)$ can be described by a set of Boolean functions $f_1,\ldots, f_n$: 
$$
x_i(t+1) = f_i(x_1(t),\ldots, x_n(t)), \quad \forall i = 1, \ldots, n.  
$$
For convenience, we let $x(t):= (x_1(t),\ldots, x_n(t)) \in \F_2^n$ be the state of the Boolean network at time~$t$, and let $$f:= (f_1, \ldots, f_n): x(t) \mapsto x(t+1).$$



\begin{definition}[Conjunctive Boolean network~\cite{jarrah2010dynamics}]\label{CBN}
	A Boolean network $f = (f_1, \ldots, f_n)$  is {\bf conjunctive} if each Boolean function $f_i$ can be expressed as follows:
	\begin{equation}\label{eq:updaterule}
	f_i(x_1,\ldots, x_n)=\prod^n_{j = 1} x_j^{\epsilon_{ji}}
	\end{equation}
	with $\epsilon_{ji}\in \{0,1\}$ for all $j=1,\ldots,n$. The associated {\bf dependency graph} is a digraph $G = (V,E)$ of $n$ vertices. An edge $v_iv_j$ exists in $G$ if and only if $\epsilon_{ij} = 1$.  
\end{definition}

\begin{remark}\label{rmk:cbntodbn}
A Boolean network is called {\em disjunctive} if each of its Boolean functions is comprised of only ``OR'' operations. 
	There is an isomorphism between
	the class of CBNs and
	the class of disjunctive Boolean networks (DBNs): Let $f$ (resp. $g$) be Boolean functions on $n$ Boolean variables
	$x_1,\ldots, x_n$, comprised of only ``AND'' (resp. ``OR'') operations.
	Then, $f(x_1, \ldots , x_n) = \bar g(\bar x_1, \ldots, \bar x_n)$, where ``$\bar{\,\,\,\,}$'' is the negation operator.  
\end{remark}

Note that a CBN uniquely determines its dependency graph and vice versa. We can thus refer a CBN to its dependency graph, i.e., we use the phrase ``a CBN $G$''.  Also, note that the dependency graph is {\em not} the state-transition graph in which vertices represent states of a CBN (hence, there are $2^n$ vertices) and directed edges indicate one-step state transitions. 
With the dependency graph, we can re-write~\eqref{eq:updaterule} as follows:
$$
x_i(t) = f_i(x(t-1)) = \prod_{v_j \in {\mathcal N}_{\rm in}(v_i)}x_j(t-1). 
$$
By recursively applying the above expression, we obtain
$
x_i(t) = \prod_{v_j \in {\mathcal N}_{\rm in}^t(v_i)}x_j(0)  
$, 
which expresses the current state of vertex $v_i$ in terms of the initial conditions of other vertices. In particular, if $x(0) \ge x'(0)$ (the inequality is entry-wise), then $x_i(t) \ge x'_i(t)$.

For the remainder of the paper, we let the dependency graph~$G$ be weakly connected. Note that if $G$ is not connected, then the results established in the paper can be applied to the connected components of $G$.  
We also note that if a vertex $v_i$ of $G$ has no incoming neighbor, then from Definition~\ref{CBN}, $\epsilon_{ji} = 0$ for all~$j = 1,\ldots, n$, and hence, from~\eqref{eq:updaterule}, 
$f_i(x_1, \ldots, x_n) = 1$. In other words, $x_i(t)\equiv 1$ for all~$t$, which implies that the Boolean variable~$x_i$ does {\em not} affect the rest at all. 
We can thus trim the size of the CBN by simply ignoring~$x_i$. More precisely, we trim the dependency graph by deleting any vertex $v_i$ that does not have an incoming neighbor, together with the edges that are incident to~$v_i$. For the above reason, we assume in the sequel that each vertex of the dependency graph $G$ has at least one incoming neighbor.

Since a CBN is a finite dynamical system, for any initial condition $x(0)\in \F_2^n$, the trajectory $x(0), x(1), \ldots$ will enter a {\em periodic orbit} in a finite number of time steps. Specifically, there exist a time $t_0 \ge 0$ and an integer number $p \ge 1$ such that $x(t_0 + p) = x(t_0)$. Further, if $x(t_0 +p') \neq x(t_0)$ for any $p' = 1,\ldots, p-1$, then the sequence $\{x(t_0), \ldots, x(t_0 + p-1)\}$, taking rotations as equivalent, is said to be a {\bf periodic orbit}, and we call $p$ its {\bf period}.  Note that a periodic orbit is also known as a {\em limit cycle}. But, to avoid confusion, we will not use ``limit cycle'' in the remainder of the paper. We will only use ``cycle'' to refer to a sequence of vertices in a dependency graph associated with a CBN.

\subsection{Universality of CBN}\label{ssec:universal}
We note here that the class of CBNs has the ``universality'' property. Specifically, to every Boolean network, there corresponds a CBN whose dynamic behavior uniquely determines the dynamic behavior of the original system. We elaborate on this fact below. 

Let $f = (f_1,\cdots, f_n)$ be the update rule of an arbitrary Boolean network on $n$ variables. By universality of logic gates~\cite{sheffer1913set}, every Boolean function $f_i$ can be expressed as a certain composition of ``NAND'' (or ``NOR'') operations. Note that $X \operatorname{nand} Y = \bar X \operatorname{or} \bar Y$. If, further, $X = X_1 \operatorname{nand}  X_2$, then $X \operatorname{nand} Y = X_1X_2 \operatorname{or} \bar Y$. Repeatedly applying the above arguments, we obtain that every Boolean function $f_i$ can be expressed as 
\begin{equation}\label{eq:generators}
  f_i(x_1,\cdots, x_n) = \operatorname{OR}^{m_i}_{j = 1} z_{i_j},   
\end{equation}
where each $z_{i_j}$ is a certain monomial in variables $x_1,\cdots, x_n$ and $\bar x_1, \ldots, \bar x_n$. Note that each $x_i$  takes only the values $1$ or $0$. Thus, $x^2_i = x_i$ and $\bar x^2_i = \bar x_i$. This, in particular, implies that there are only finitely many different monomials. For example, one can reduce a monomial $x^2_i\bar x^3_j$ to $x_i\bar x_j$. Also, note that $x_i\bar x_i = 0$ for any $i= 1,\ldots, n$. We call a monomial $z$ trivial if it contains any such product.  We then let $Z$ be the collection of all reduced,  nontrivial monomials.

Recall that a Boolean network is disjunctive if each of its Boolean functions is comprised of only ``OR'' operations. Moreover, disjunctive Boolean networks (DBNs) one-to-one correspond to CBNs.  
The update rule~\eqref{eq:generators} induces a DBN over the collection $Z$ of monomials (i.e., the state variables are monomials) as follows: For a monomial $z = \prod^p_{s = 1}x_{i_s}\prod^q_{t = 1} \bar x_{i_t}$, we define its update rule as follows:
\begin{equation}\label{eq:defgz}
g_z:= \prod^p_{s = 1}f_{i_s} \prod^q_{t = 1}\bar f_{j_t}.
\end{equation}
Note that the right hand side of~\eqref{eq:defgz} is disjunctive in monomials of $Z$. To see this, first note that $\bar f_i = \prod^{m_i}_{j = 1} \bar z_{i_j}$. Also, note that $(X_1 \operatorname{or} X_2)Y = X_1Y \operatorname{or} X_2Y$. 

We have thus constructed a DBN from the original Boolean network. Moreover, by its construction, the dynamic behavior of the DBN uniquely determines the dynamic behavior of the original system. Then, by the isomorphism between the two classes of CBNs and DBNs  (Remark~\ref{rmk:cbntodbn}), 
one is able to translate a DBN into a CBN. 
We further note that the above approach of introducing higher order moments of state variables is essentially similar to the the approch of Carleman linearizaion~\cite{carleman1932application} by which one approximates a nonlinear differential equation by a linear system with an enlarged number of states.

\section{Reduced systems and induced dynamics}\label{reduction}
We introduce in the section the notion of a {\em reduced system} associated with a CBN. A reduced system is a significantly simplified system, which exhibits almost the same asymptotic behavior as the original CBN. The section is divided into three parts: In Subsection~\ref{ssec:3sc}, we review a few facts about the asymptotic behavior of a strongly connected CBN. Then, in Subsection~\ref{ssec:thm1}, we define precisely a reduced system, and state the first main result of the paper (Theorem~\ref{pro:dynamics}), which relates the asymptotic behavior of the reduced system to the asymptotic behavior of the original system. Subsection~\ref{ssec:pfthm1} is then devoted to the proof of Theorem~\ref{pro:dynamics}.          

\subsection{Review of strongly connected CBNs}\label{ssec:3sc}
Let $G$ be strongly connected, and $\{C_1, \ldots, C_m\}$ be the collection of cycles in $G$, and $n_i$ be the length of the cycle $C_i$ for $i = 1,\ldots, m$. Let $p$ be the greatest common divisor (gcd) of the~$n_i$'s:
$$
p := {\rm gcd}\{n_1,\ldots, n_m\}.
$$ 
In the case $G$ is a single vertex without a self-arc, we set $p = 0$. The integer~$p$ is also known as the {\em loop number} of the graph~$G$.  
We have the following fact:

\begin{lemma}\label{lem:periodofsg}
	Let $G$ be strongly connected, and the loop number $p$ be positive. 
	Then, $p'$ is the period of a periodic orbit of the CBN $G$ if and only if $p'$ divides~$p$.     
\end{lemma}

We refer to~\cite{jarrah2010dynamics,gao2016periodic} for a proof of Lemma~\ref{lem:periodofsg}.  We now fix the graph~$G$, and introduce an equivalence relation defined on the vertex set~$V$ of~$G$. A vertex $v_i$ is said to be related to $v_j$, or simply written as $v_i \sim_{p} v_j$, if there exists a path from $v_i$ to $v_j$ whose length is a multiple of~$p$. We have shown in~\cite{stabilityfull} that the relation ``$\sim_{p}$'' defined above is an equivalence relation. For a vertex $v$ of $G$, we denote by $[v]$ the equivalence class that contains~$v$. 
Now, let two vertices $v_i$ and $v_j$ be in the same equivalence class. Then, it is known that the length of any walk from~$v_i$ to~$v_j$ is a multiple of~$p$ (see, for example,~\cite{stabilityfull}). An immediate consequence is then the following: 

\begin{lemma}\label{lem:partition}
	Fix a vertex $v_0$ of $G$, and choose vertices $v_1,\ldots, v_{p-1}$ such that $v_i\in \mathcal{N}_{out}(v_{i-1})$ for all $i = 1,\ldots, p-1$. Then, the subsets $[v_0],\ldots, [v_{p-1}]$  form a partition of $V$. Moreover, for any $i = 0, \ldots, p-1$,  
	\begin{equation}\label{eq:inandout}
	\left\{
	\begin{array}{l}
	{\mathcal N}_{\rm out}([v_i]) = [v_{i+1 \bmod  p}],  \\
	{\mathcal N}_{\rm in}([v_i]) = [v_{i-1 \bmod  p}]. 
	\end{array}
	\right.
	\end{equation}\,
\end{lemma}

%

For a subset $V'\subseteq V$, we let $x_{V'}(t)\in \F^{|V'|}_2$ be the restriction of the state~$x(t)$ to~$V'$. For example, if $V' = \{v_i,v_j\}$, then $x_{V'}(t) = (x_i(t),x_j(t))$.  
We now state in the following lemma a few properties  about periodic orbits of a strongly connected CBN.

\begin{lemma}\label{lem:periodicorbit}
	Let $G$ be strongly connected and the loop number $p$ be positive. Let $x(0)$ be an initial condition of a CBN~$G$. Then, there exists a time step $N$, divisible by $p$,  such that the following hold:
	\begin{enumerate}
		\item For any $i = 0,\ldots, p-1$, $$x_{[v_i]}(N) = \prod_{v_{k} \in [v_i]} x_{k}(0){\bf 1},$$
		where ${\bf 1}$ is a vector of $1$'s of an appropriate dimension. 
		\item For any $t \ge  N$ and any $i = 0, \ldots, p-1$, there is a $y_i(t)\in \F_2$ such that  $x_{[v_i]}(t) = y_i(t) \bf {1}$. Moreover,
		$$
		y_{i}(t + 1) = y_{(i-1) \bmod p}(t).
		$$
	\end{enumerate}\,
\end{lemma}


\begin{Remark}
	The first item of Lemma~\ref{lem:periodicorbit} says that if there exists an entry of $x_{[v_i]}(0)$ holding the value~$0$, then all the entries of $x_{[v_i]}(N)$ will hold the value~$0$. The second item of Lemma~\ref{lem:periodicorbit}, together with the second item of Lemma~\ref{lem:partition}, implies that the state $x(N)$ is in a periodic orbit (though $N$ is not necessarily the minimum integer for the CBN to enter the periodic orbit). 
\end{Remark}

We refer the reader to Section~3 in~\cite{stabilityfull} for analyses and a proof of the above lemma.
From item~2 of Lemma~\ref{lem:periodicorbit},  the periodic orbits of the CBN~$G$ correspond one-to-one to the binary necklaces of length~$p$. To see this,  we first let $s = \alpha_0\ldots \alpha_{p-1}$ be a binary necklace, with $\alpha_i$ being either~$0$ or~$1$. Then, the periodic orbit associated with~$s$ is given as follows: we first define a state $x\in \F^n_2$ by setting $x_{[v_i]} = \alpha_i{\bf 1}$, for $i = 0,\ldots, p-1$. Then, from Lemma~\ref{lem:periodicorbit}, the state $x$ is in a periodic orbit (as one can let $N = 0$). Note that if the order of the binary necklace $s$ is $q$ (which divides~$p$), then the corresponding periodic orbit is simply the sequence $\{x, f(x), \ldots, f^{q-1}(x)\}$, taking rotations as equivalent. Conversely, given a periodic orbit, one can first pick a state~$x$ out of the orbit, and then define a binary necklace $x_0x_1\ldots x_{p-1}$, with the $x_i$'s being the states of the vertices $v_i$'s.

\subsection{Reduced systems}\label{ssec:thm1}


We now consider a weakly connected CBN~$G$.  By applying the strong component decomposition (see Definition~\ref{def:SCD}), we obtain its strongly connected components (or simply strong components) $G_1,\ldots, G_q$. An example of weakly connected digraph is provided in Fig.~\ref{original}. We further denote by $p_i$ the loop number of $G_i$, for $i = 1,\ldots, q$. If $p_i > 0$, then within  each strongly connected component $G_i = (V_i, E_i)$,  we can define the equivalence relation $``\sim_{p_i}"$  on its vertex set~$V_i$. From Lemma~\ref{lem:partition}, by choosing an arbitrary vertex $v_{i_0} \in V_i$, together with the following vertices: 
$$v_{i_1} \in \mathcal{N}_{\rm out}(v_{i_0}; G_i),\ldots, v_{i_{p_i - 1}} \in \mathcal{N}_{\rm out}(v_{i_{p_i-2}};G_i),$$ 
we obtain 
equivalence classes $[v_{i_j}]$'s, for $j = 0,\ldots, p_i - 1$, which partition the vertex set $V_i$. If $p_i = 0$, then $G_i$ is comprised only of a single vertex without a self-arc. For consistency, we denote by~$v_{i_0}$ the single vertex and $[v_{i_0}]$ the singleton $\{v_{i_0}\}$.  For the remainder of the paper, we fix these choices of the vertices $v_i$'s, $i = 0,\ldots, p-1$. 


\begin{figure}[ht]
	\centering
	\includegraphics[height=54mm]{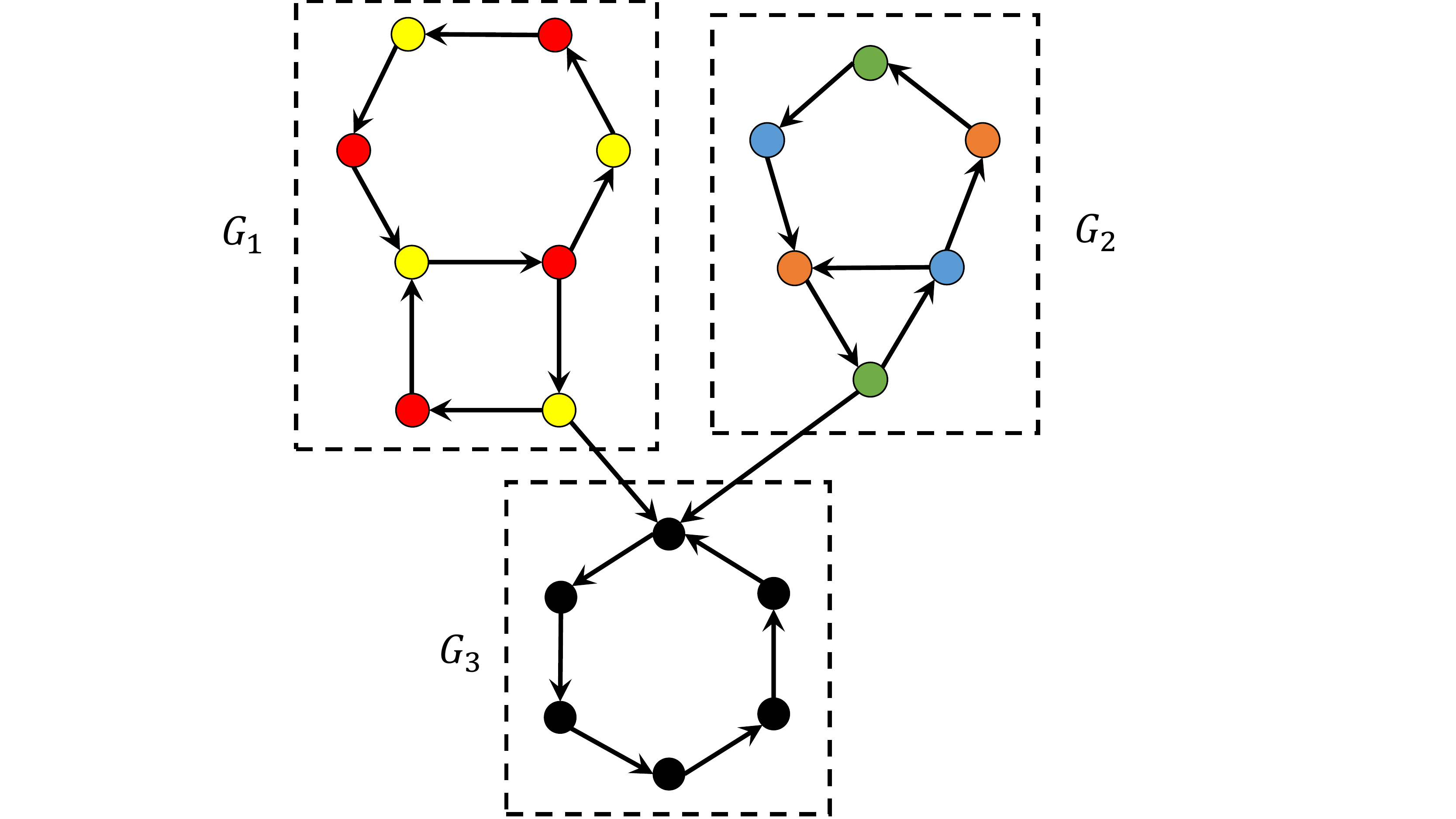}
	\caption{The digraph in the figure is weakly connected, consisting of three strongly connected components, denoted by $G_1,G_2,G_3$. In this example, $p_1=2$, $p_2=3$, and $p_3=6$. Vertices in the same color belong to the same equivalence class. We have that $\cS^0(G)=\{G_1,G_2\}$ and $\cS^1(G)=\{G_3\}$. }
	\label{original}
\end{figure}

We construct the associated reduced system (as a CBN) by defining its dependency graph, denoted by  $H = (U, F)$. The digraph $H$ is weakly connected, comprised of~$q$ strongly connected components, i.e., the numbers of strongly connected components of $G$ and of $H$ are the same. Denote by $H_i = (U_i, F_i)$ for $i = 1, \ldots,q$, the strongly connected components of~$H$. Each $H_i$ is a cycle of length $p_i$, with   
$$
\left\{
\begin{array}{l}
U_i := \{u_{i_0},\ldots, u_{i_{p_i-1}}\}, \\     
F_i := \{u_{i_j} u_{i_{(j+1) \bmod p_i}} \mid j = 0,\ldots, p_i-1\}.
\end{array}
\right. 
$$
In the case $p_i = 0$, we simply have $U_i = \{u_{i_0}\}$ and $F_i = \varnothing$. 
We provide an example of reduced system in Fig.~\ref{reduced}.
In this way, the vertices $u_{i_j}$'s of $H_i$ correspond one-to-one to the equivalence classes $[v_{i_j}]$'s of $G_i$. Moreover, if $p_i > 0$, then the out-neighbor (resp. in-neighbor)  of $u_{i_j}$ in $H_i$ is $u_{i_{(j+1) \bmod p_i}}$ (resp. $u_{i_{(j -1) \bmod p_i}}$), which is consistent with~\eqref{eq:inandout}.  Now, to determine the digraph $H$, it suffices to specify the edges that connect the cycles $H_i$'s. Let $u_{i_j}$ and $u_{i'_{j'}}$, with $i\neq i'$,  be vertices of $H_i$ and of $H_{i'}$, respectively.  Then, $u_{i_j} u_{i'_{j'}}$ is an edge of $H$ if there exists an edge $v'v''$ in $G$ with $v'\in [v_{i_j}]$ and $v''\in [v_{i'_{j'}}]$. The construction of the digraph $H$ is now complete.

\begin{definition}[Reduced system]\label{def:reducedsys}
	 Given a weakly connected CBN $G$, we call a CBN the {\bf reduced system} if its dependency graph is given by the above constructed digraph~$H$. 
\end{definition}

The size of the digraph $H$ is, in general, much smaller than the size of the original dependency graph $G$. An illustrative example is given in Fig.~\ref{reduced}. Nevertheless, we will see soon that the reduced system CBN~$H$ encodes all the information that is needed for determining the asymptotic behavior of the original CBN~$G$.  

 We provide below a few facts about the computational complexity of the above construction of a reduced system. The construction process is comprised of four major steps: 
\begin{enumerate}[label=({S}{\arabic*})]
    \item Find strong components $G_i = (V_i, E_i)$ of $G$. \label{s1} 
    \item Find the loop number $p_i$ for each strong component $G_i$. This can be done by finding all cycles of each  $G_i$.\label{s2}
    \item Partition each $E_i$ into equivalence classes (with respect to the equivalence relation $\sim_{p_i}$). The strong components $H_i$ (cycles) are then constructed. \label{s3} 
    \item Construct the edges that connect vertices in different $H_i$.\label{s4}
\end{enumerate}
 For~\ref{s1}, the Kosaraju-Sharir algorithm~\cite{sharir1981strong} can be used and its time-complexity is $\mathcal{O}(|V|+|E|)$.  For~\ref{s2}, one can use Johnson's algorithm~\cite{johnson1975finding} to find all cycles of a given strongly connected component. Its time-complexity is  $\mathcal{O}((|V_i|+|E_i|)C_i)$, where $C_i$ is the number of cycles in $G_i$. Since $G_i$ is a subgraph of $G$, $C_i$ is upper bounded by the number of cycles in $G$, which we now denote by~$C$. Thus, the total time-complexity of finding all cycles in all strong components is given by $\mathcal{O}(\sum_{i=1}^q(|V_i|+|E_i|)C) = \mathcal{O}((|V|+|E|)C)$. For~\ref{s3}, the time-complexity is  simply $\mathcal{O}(\sum^q_{i = 1}|V_i|) = \mathcal{O}(|V|)$. Finally, for~\ref{s4}, the time-complexity is $\mathcal{O}(|E|)$. Thus, the time-complexity of~\ref{s2} is dominating over the other steps.  We then conclude that the overall time-complexity of constructing a reduced system is given by $\mathcal{O}((|V|+|E|)C)$.

\begin{figure}[ht]
	\centering
	\includegraphics[height=49mm]{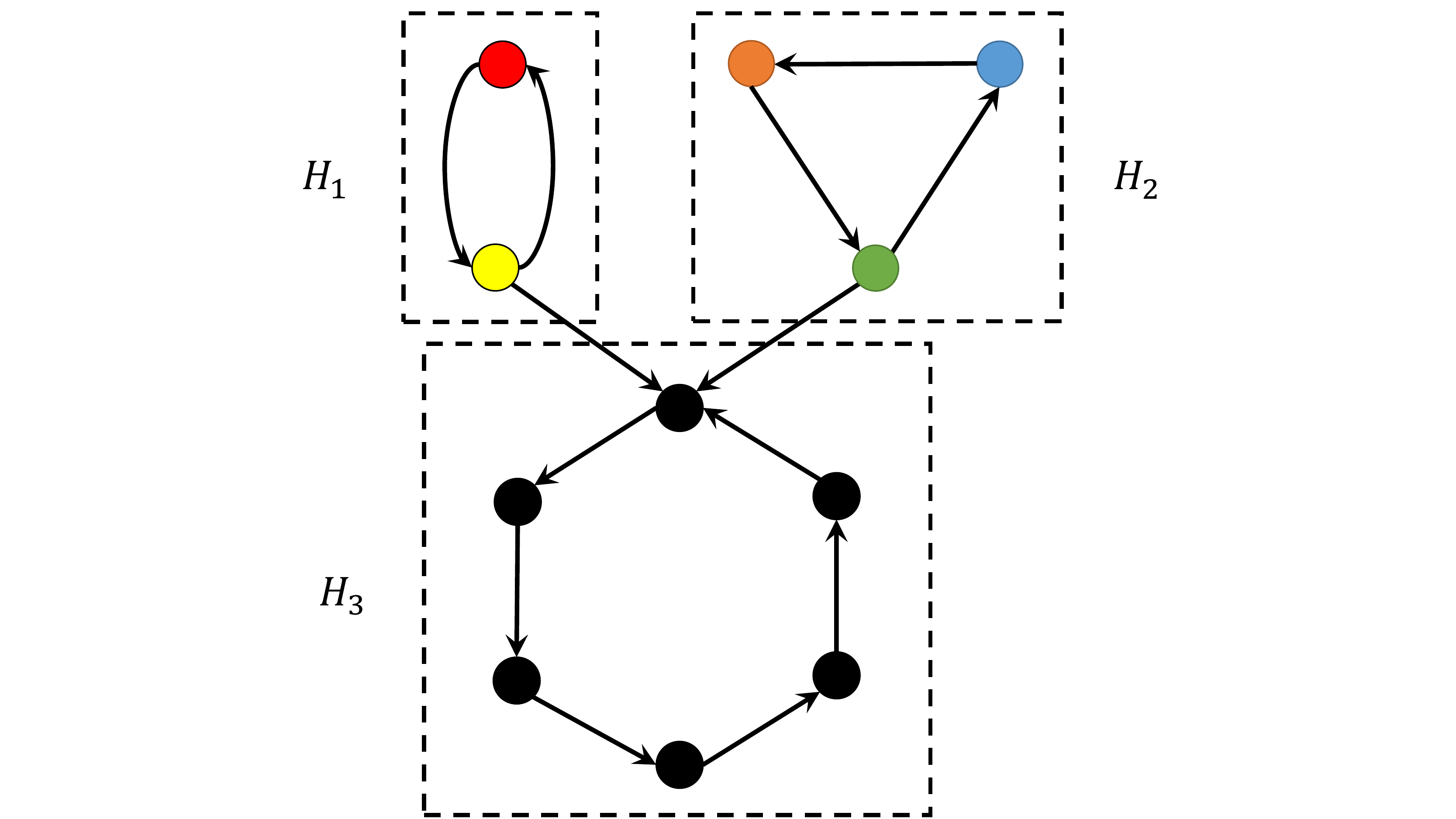}
	\caption{This figure shows the dependency graph of the reduced system $H$ of the system shown in Fig.~\ref{original}. As shown in the figure, $G_1$ and $G_2$ are reduced to cycles of length $2$ and~$3$, respectively.} 
	\label{reduced}
\end{figure}

Let $\cS(H):=\{H_1,\ldots, H_q\}$ be the collection of the strongly connected components of $H$. We impose the same partial order on $\cS(H)$ as we did for $\cS(G)$. 
We  further partition $\cS(H)$ into $$\cS(H) = \sqcup^L_{l = 0}\cS^l(H),$$ 
where $\cS^0(H)$ is the collection of the maximal elements and each $\cS^{l}(H)$ is the immediate successor of $\cS^{l-1}(H)$ for $l = 1,\ldots, L$. 
Note that from our construction of~$H$, we have that $H_i\succ H_j$ if and only if $G_i \succ G_j$, and hence $H_i\in \cS^l(H)$ if and only if $G_i \in \cS^l(G)$. 
We also recall an earlier assumption that each vertex of the digraph~$G$ has at least one incoming neighbor. So, if $G_i \in \cS^0(G)$, then $p_i > 0$, and hence the length of the cycle $H_i$ is positive.

We now relate the asymptotic behavior of a CBN~$G$ to the asymptotic behavior of its reduced system~$H$. In the remainder of the section, we will use~$x$ (resp. $y$) to denote the state of the CBN~$G$ (resp. $H$). Let $x(0)$ be an initial condition of $G$. An {\em induced initial condition} $y(0)$ of $H$ is defined as follows: Let $y_{i_j}(0)$ be the initial condition of the vertex $u_{i_j}$, and define 
\begin{equation}\label{eq:initialcondition}
y_{i_j}(0) := \prod_{v_k \in [v_{i_j}]} x_{k}(0).
\end{equation}
In other words, $y_{i_j}(0)$ is $1$ if and only if all the $x_k(0)$'s, for $v_k\in[v_{i_j}]$, are~$1$. With the definitions and notations above, we now state the first main result of the paper:

\begin{theorem}\label{pro:dynamics}  
	Let $x(0)$ be an initial condition of a weakly connected CBN~$G$, and $y(0)$ be the induced initial condition of its reduced system~$H$. Then, there exists a time step $N$ such that $x(N)$ and $y(N)$ are in periodic orbits of $G$ and $H$, respectively. Moreover, for any $t \ge N$, the following hold: 
	\begin{enumerate}    
		\item For any strongly connected component $G_i$,  
		\begin{equation}\label{eq:veryimportant}
		x_{[v_{i_j}]}(t) = y_{i_j}(t) {\bf 1}, \quad \forall j = 0,1,\ldots,p_i - 1. 
		\end{equation}
		\item	If the loop number $p_i$ of $G_i$ is positive, then 
		\begin{equation}\label{eq:lem2thm1}
		y_{i_j}(t + 1) =  y_{i_{(j - 1) \bmod p_i}}(t), \quad \forall j = 0,1,\ldots,p_i - 1.
		\end{equation}
	\end{enumerate}\,
\end{theorem}

\begin{Remark}
	Since the vertices $u_{i_j}$'s of $H$ one-to-one correspond to the equivalence classes $[v_{i_j}]$'s,  
	the first item of Theorem~\ref{pro:dynamics} implies that the asymptotic behavior of a CBN can be uniquely determined by the asymptotic behavior of its reduced system. The second item implies that in a periodic orbit, the dynamics $x_{V_i}(t)$ proceeds as if $G_i$ was a disjoint strongly connected component (compared with the second item of Lemma~\ref{lem:periodicorbit}).
\end{Remark}

Since the period of a periodic orbit of $x(t)$ is completely determined by the $x_{V_i}(t)$'s, for $p_i> 0$, we have the following fact as an immediate consequence of Theorem~\ref{pro:dynamics}.  

\begin{corollary}\label{lem3:thm1}
	Let $N_G > 0$ be the least common multiple of the loop numbers $p_i$. Then, the period of any periodic orbit of system $G$ divides $N_G$. 
\end{corollary}


\subsection{Analyses and Proof of Theorem~\ref{pro:dynamics}}\label{ssec:pfthm1} 
We now have a sequence of lemmas that lead to the proof of Theorem~\ref{pro:dynamics}. We first deal with the special case where $x(0)$ is already in a periodic orbit of system~$G$ (Lemmas~\ref{lem1:thm1} and~\ref{lem2:thm1}). We then extend the results to a general case using the facts established in Lemmas~\ref{lem4:thm1} and~\ref{lem5:thm1}. We start with the following fact:     


\begin{lemma}\label{lem1:thm1}
	Let $x(0)$ be in a periodic orbit of system $G$. Then,~\eqref{eq:veryimportant} holds for any $t \ge 0$ (i.e., $N = 0$). 
\end{lemma}

\begin{proof}
	We first show that if the two vertices $v_k$ and $v_l$ belong to the same equivalence class $[v_{i_j}]$ in the strongly connected component $G_i$, then $x_{k}(t) = x_{l}(t)$. 
	
	Suppose to the contrary that there exists a time step $t_0 \ge 0$ such that $x_k(t_0) \neq x_l(t_0)$. Without loss of generality, we let $x_k(t_0) = 0$ and $x_l(t_0) = 1$. Let $\sigma_i(t_0)$ be the total number of $0$'s held by the vertices in $G_i$ at time $t_0$. Since $x(t_0)$ is in a periodic orbit, if we let $T$ be the period, then $\sigma_i(t_0) = \sigma_i(t_0 + mT)$ for any $m \ge 0$. 
	
	Now, consider an auxiliary CBN whose dependency graph is $G_i$. Let a state $z(t_0)$ of $G_i$ be defined as $z(t_0):= x_{V_i}(t_0)$. 
	Appealing to Lemma~\ref{lem:periodicorbit}, we obtain a positive integer $N'$, as a multiple of $p_i$, such that $z(t_0+ N')$ is in a periodic orbit of system $G_i$, and moreover, 
	$$z_{[v_{i_j}]}(t_0 + m N') = \prod_{v_k\in [v_{i_j}]} z_k(t_0){\bf 1}, \quad \forall m \ge 1.$$
	Let $\sigma'_i(t)$, for $t \ge t_0$, be the total number of $0$'s in $G_i$.  
	Since $z_k(t_0) = 0$ and $z_l(t_0) = 1$, and $v_k$, $v_l$ belong to the same equivalence class $[v_{i_j}]$, we have
	$\sigma'_i(t_0) < \sigma'_i(t_0 + N') = \sigma'_i(t_0 + mN')$ for any $m > 0$.  
	
	On the other hand, for any vertex~$v_j$ in $G_i$ and any $t \ge t_0$, we have ${\mathcal N}^{t - t_0}_{\rm in}(v_j; G_i) \subseteq  {\mathcal N}^{t - t_0}_{\rm in}(v_j)$. 
	Thus, 
	$$
	z_j(t) = \prod_{v_k\in {\mathcal N}^{t - t_0}_{\rm in}(v_j; G_i)} z_k(t_0) 
	\ge \prod_{v_k\in {\mathcal N}^{t - t_0}_{\rm in}(v_j)} x_k(t_0) = x_j(t),
	$$
	which, in particular, implies that $\sigma_i(t) \ge \sigma'_i(t)$. But then, 
	\begin{align*}
	\sigma_i(t_0 + TN') = \sigma_i(t_0) &= \sigma'_i(t_0) \\
	&< \sigma'_i(t_0 + TN')\le \sigma_i(t_0 + TN'), 
	\end{align*}
	which is a contradiction. We have thus shown that $x_k(t) = x_l(t)$ if $v_k$ and $v_l$ belong to the same equivalence class. 
	
	Next, we show that~\eqref{eq:veryimportant} holds. The proof will be carried out by induction on time step~$t$. For the base case $t = 0$,~\eqref{eq:veryimportant} holds by~\eqref{eq:initialcondition} and the fact that the vertices in the same equivalence class hold the same value. For the inductive step, we assume that~\eqref{eq:veryimportant} holds for $t$, and we prove for $(t + 1)$. It suffices to show that $x_{k}(t + 1) = y_{i_j}(t + 1)$ for some (and hence any) $v_k\in [v_{i_j}]$. We first show that if $y_{i_j}(t + 1) = 1$, then $x_k(t + 1) = 1$. Using the induction hypothesis, we have
	$$
	y_{i_j}(t + 1) = \prod_{u_{i'_{j'}}\in {\mathcal N}_{\rm in}(u_{i_j})} y_{i'_{j'}}(t) = \prod_{v_l\in {\mathcal N}_{\rm in}([v_{i_j}])} x_l(t).
	$$
	For any $v_k \in [v_{i_j}]$, we have ${\mathcal N}_{\rm in}(v_k)\subseteq {\mathcal N}_{\rm in}([v_{i_j}])$, and hence
	$$
	1 = y_{i_j}(t + 1) \le \prod_{v_l\in {\mathcal N}_{\rm in}(v_k)} x_l(t) = x_k(t + 1),
	$$
	which implies that $x_k(t + 1) = 1$. We next show that if $y_{i_j}(t + 1) = 0$, then $x_k(t + 1) = 0$. Let $u_{i'_{j'}}\in {\mathcal N}_{\rm in}(u_{i_j})$ be such that $y_{i'_{j'}}(t) = 0$. We further let $v_k$ and $v_l$ be two vertices of $G$ such that $v_k\in [v_{i_j}]$, $v_l\in [v_{i'_{j'}}]$, and $v_l v_k$ is an edge. Then, $v_l\in {\mathcal N}_{\rm in}(v_k)$, and moreover,  by the induction hypothesis, $x_l(t) = y_{i'_{j'}}(t) = 0$. We thus conclude that $x_k(t + 1) = 0$. This completes the proof.   
\end{proof} 

For each $G_i$ of positive loop number $p_i$, we let $\#_i(t):= \sum^{p_i - 1}_{j = 0}  y_{i_j}(t)$. The following result establishes the second item of Theorem~\ref{pro:dynamics}:

\begin{lemma}\label{lem2:thm1}
	Let $x(0)$ be in a periodic orbit of system $G$. Then, for each $G_i$ with positive $p_i$, $\#_i(t)$ is constant. Thus, for any $t \ge 0$,     
	\begin{equation}
	y_{i_j}(t + 1) =  y_{i_{(j - 1) \bmod p_i}}(t), \quad j = 0, 1, \ldots, p_i -1. \nonumber
	\end{equation}\,
\end{lemma}

\begin{proof}
	For any $t\ge 0$, we have
	$
	y_{i_j}(t + 1) = \prod_{u_k \in {\mathcal N}_{\rm in}(u_{i_j})} y_{k}(t) 
	$. 
	Since $p_i$ is positive, $u_{i_{(j-1) \bmod p_i}}$ is in ${\mathcal N}_{\rm in}(u_{i_j})$. Thus, if $y_{i_{(j - 1)\bmod p_i}}(t) = 0$, then $y_{i_j}(t + 1) = 0$. The above arguments imply that $\#_i(t)$, which is the number of~$1$'s possessed by the vertices of $H_i$, is nonincreasing in~$t$. From Lemma~\ref{lem1:thm1}, $y(0)$ is in a periodic orbit of $H$. It thus follows that $\#_i(t) = \#_i(0)$ for any $t\ge 0$, and hence~\eqref{eq:lem2thm1} has to hold. 
\end{proof}

We have established Theorem~\ref{pro:dynamics}, but only for the case where $x(0)$ is in a periodic orbit. 
We now extend the results to a general case where $x(0)$ is not necessarily in a periodic orbit of system~$G$. 

Lemma~\ref{lem1:thm1} implies that if $x(0)$ is in a periodic orbit of system $G$, then $y(0)$ is in a periodic orbit of the reduced system $H$ of the same period. Conversely, we have the following fact:

\begin{lemma}\label{lem4:thm1}
	Suppose that for each equivalence class $[v_{i_j}]$, $x_{[v_{i_j}]}(0) = y_{i_j}(0){\bf 1}$, and moreover, $y(0)$ is in a periodic orbit of system $H$; then,  $x(0)$ is in a periodic orbit of system $G$. 
\end{lemma}

\begin{proof}
	It suffices to show that for any $t \ge 0$, the vertices that belong to the same equivalence class hold the same value.  If this holds, then we can apply the same arguments as in the proof of Lemma~\ref{lem1:thm1} and obtain that $x_{[v_{i_j}]}(t) = y_{i_j}(t) {\bf 1}$ for any $t \ge 0$ and any equivalence class $[v_{i_j}]$.  
	
	Suppose not, and let $(t_0 + 1)$, for $t_0 \ge 0$, be the first time step such that there exist two vertices $v_k$ and $v_l$ in an equivalence class $[v_{i_j}]$ such that $x_k(t_0 + 1) \neq x_l(t_0 + 1)$. Without loss of generality, we assume that $x_k(t_0 + 1) = 0$ and $x_l(t_0 + 1) = 1$. 
	Since the equivalence class $[v_{i_j}]$ contains at least two vertices $v_k$ and $v_l$, the loop number $p_i$ of $G_i$ is positive, and hence the equivalence class $[v_{i_{(j - 1) \bmod p_i}}]$ exists. Furthermore, since 
	${\mathcal N}_{\rm in}([v_{i_j}]; G_i) = [v_{i_{(j-1) \bmod p_i}}]$, we must have $x_{[v_{i_{(j - 1) \bmod p_i}}]}(t_0) = \bf{1}$ because otherwise, $$x_{[v_{i_j}]}(t_0 + 1)  = x_{[v_{i_{(j - 1) \bmod p_i}}]}(t_0) = {\bf 0}.$$ 
	Correspondingly, we have $y_{i_{(j - 1) \bmod p_i}}(t_0) = 1$. 
	
	On the other hand, since $x_{k}(t_0 + 1) = 0$, there must exist a vertex $v_{k'} \in {\mathcal N}_{\rm in}(v_k) - [v_{i_{(j-1)\bmod p_i}}]$ such that $x_{k'}(t_0) = 0$. We let $[v_{i'_{j'}}]$ be the equivalence class that contains $v_{k'}$. By construction of~$H$, we have $u_{i'_{j'}} \in {\mathcal N}_{\rm in}(u_{i_j})$. From the assumption on~$t_0$, we have $x_{[v_{i'_{j'}}]}(t_0) = {\bf 0}$ and $y_{i'_{j'}}(t_0) = 0$. It then follows that $y_{i_j}(t_0 + 1) = 0$. But then, from~\eqref{eq:lem2thm1} and the fact that $y(0)$ is in a periodic orbit of system $H$, we have 
	$y_{i_{(j - 1) \bmod p_i}}(t_0) = 0$, which is a contradiction.
\end{proof}

To proceed, we define a map $\rho: \F_2^{|G|} \to \F_2^{|G|}$ as follows: For a given state $x$ and an equivalence class $[v_{i_j}]$, we let 
$$\rho(x)_{[v_{i_j}]}:= \prod_{v_k \in [v_{i_j}]} x_{k} {\bf 1}.$$  
We then recursively define a sequence $\xi_t\in \F_2^{|G|}$, for $t \ge 0$, as follows: For the base case $t = 0$, we let $\xi_0:= \rho(x(0))$. For the inductive step, we let $\xi_{t+1} := \rho(f(\xi_t))$, i.e., we first let the dynamics of system $G$ proceed one time step, with $\xi_t$ the current state, to obtain $f(\xi_t)$, and then apply the map $\rho$ to obtain $\xi_{t + 1}$.  

\begin{lemma}\label{lem5:thm1}
	The following hold for the sequence $\{\xi_t\}_{t \ge 0}$: 
	\begin{enumerate}
		\item	For any $t \ge 0$ and any equivalence class $[v_{i_j}]$, $\xi_{t, [v_{i_j}]} = y_{i_j}(t) {\bf 1}$. 
		\item For a given but arbitrary $s \ge 0$, we let $x'(s):= \xi_s$. Then, there exists a time step $N'_s \ge s$ such that $x(t) = x'(t)$ for any $t \ge N'_s$, i.e., the two trajectories $x(t)$ and $x'(t)$, for $t \ge s$, will eventually be the same.  
	\end{enumerate}\,
\end{lemma}


\begin{proof}
	The first item can be established by induction on~$t\ge 0$. 
	We omit the proof since it is similar to the proof of Lemma~\ref{lem1:thm1}. We prove below the second item. The proof is carried out by induction on $s \ge 0$. 
	
	{\em Base case: $s = 0$.} 
	We choose $N'_0$ sufficiently large so that both $x(N'_0)$ and $x'(N'_0)$ are in periodic orbits of periods $T$ and $T'$, respectively. Now, for each $v_i$, we write 
	$$
	\left\{
	\begin{array}{l}
	x_i(N'_0) = \prod_{v_j\in {\mathcal N}_{\rm in}^{N'_0}(v_i)}x_j(0), \\
	x_i'(N'_0) = \prod_{v_j\in {\mathcal N}_{\rm in}^{N'_0}(v_i)} x'_j(0).
	\end{array}
	\right. 
	$$ 
	Since $x(0) \ge x'(0)$ (entry-wise), we obtain $x_i(N'_0) \ge x'_i(N'_0)$. We now show that $x_i(N'_0) = x'_i(N'_0)$. It suffices to show that if $x'_i(N'_0) = 0$, then $x_i(N'_0) = 0$.    
	
	Suppose to the contrary that $x'_i(N'_0) = 0$ but $x_i(N'_0) = 1$; then, there must be a vertex $v_a\in {\mathcal N}_{\rm in}^{N'_0}(v_i)$ such that $x'_a(0) = 0$ and $x_a(0) = 1$. By the construction of $x'(0)$, there is a vertex $v_{a'}\in [v_a]$ such that $x_{a'}(0) = 0$. Without loss of generality, we assume that $v_a$ is a vertex of $G_k$. Since $[v_a]$ contains at least two vertices $v_a$ and $v_{a'}$, the loop number $p_k$ of $G_k$ is positive.  Appealing to the proof of Lemma~\ref{lem1:thm1}, we obtain a time step $lp_k$ with $l> 0$ such that $x_{[v_a]}(l'p_k) = 0$ for all $l' \ge l$. 
	
	Now, let $N''_0 := N'_0 + lp_kT$.  
	For ease of notation, we let $\Delta:= N''_0 - lp_k$, and write 
	$$x_i(N''_0) = \prod_{v_j\in {\mathcal N}_{\rm in}^{\Delta}(v_i)} x_j(lp_k).$$ 
	Since $(N''_0 - N'_0)$ is a multiple of the period~$T$, we have $x_i(N''_0) = x_i(N'_0) = 1$. On the other hand, since $(\Delta - N'_0) = lp_k(T - 1)$ as a multiple of $p_k$, the set ${\mathcal N}_{\rm in}^{\Delta}(v_i)$ intersects the equivalence class $[v_a]$; indeed, since $v_a\in {\mathcal N}^{N'_0}_{\rm in}(v_i)$ and $G_k$ is strongly connected of positive loop number $p_k$, ${\mathcal N}^{N'_0 + l'p_k}_{\rm in}(v_i)$ intersects $[v_a]$ for any $l' \ge 0$. But then, since $x_{[v_a]}(lp_k) = {\bf 0}$, we have $x_i(N''_0) = 0$, which is a contradiction. 
	
	{\em Inductive step.} We assume that the lemma holds for $s \ge 0$, and prove for $(s + 1)$. 
	From the induction hypothesis, there exists an integer $N'_s \ge s$ such that $x(t) =  x'(t)$ for any $t \ge N'_s$. Now, let $x''(s + 1):= \xi_{s + 1}$. Then, one can treat $x'(s + 1)$ and $x''(s + 1)$ as two ``initial'' conditions of system $G$ (compared with $x(0)$ and $x'(0)$ in the base case), and obtain a time step $N'_{s + 1} \ge (s + 1)$ such that $x'(t) = x''(t)$ for any $t \ge N'_{s + 1}$. Without loss of generality, we can assume that $N'_{s + 1} \ge N'_s$. Then,  
	$
	x(t) = x'(t) = x''(t)
	$ for any $t \ge N'_{s + 1}$. 
\end{proof}

With the lemmas above, we are now in a position to prove Theorem~\ref{pro:dynamics}:

\begin{proof}[Proof of Theorem~\ref{pro:dynamics}]
	We choose $s\ge 0$ such that $y(s)$ is in a periodic orbit of the reduced system $H$. Then, from Lemma~\ref{lem4:thm1} and the first item of Lemma~\ref{lem5:thm1}, we have that $\xi_s$ is in a periodic orbit of system $G$. Moreover, from Lemma~\ref{lem1:thm1}, if we let $x'(s):= \xi_s$, then for any $t \ge s$ and any equivalence class $[v_{i_j}]$, 
	$x'_{[v_{i_j}]}(t)  = y_{i_j}(t) {\bf 1}$. 
	Further, from the second item of Lemma~\ref{lem5:thm1}, there exists a time step $N'_s\ge s$  such that for any $t \ge N'_s$ and any equivalence class $[v_{i_j}]$,  $x_{[v_{i_j}]}(t) = x'_{[v_{i_j}]}(t)$. The proof is then complete by setting $N:= N'_s$. 
\end{proof}

\section{Asymptotic behavior of a reduced system} The section deals with the asymptotic behavior of a reduced system, and is divided into several parts. In Subsection~\ref{ssec:elem}, we introduce a special class of CBNs which are defined over a simple class of digraphs, termed elementary digraphs (Definition~\ref{def:elemdigraph}), and characterize their asymptotic behaviors. We show in Subsections~\ref{ssec:simpleunion} and~\ref{ssec:general}, how these CBNs serve as the building blocks which lead to a complete characterization of the asymptotic behavior of a general CBN. The main result is stated in Theorem~\ref{maintheorem}, and the proof is given in Subsection~\ref{ssec:proofthm2}.

\subsection{Elementary digraphs}\label{ssec:elem}
We first have the following definition: 

\begin{definition}[Elementary digraph]\label{def:elemdigraph}
	A weakly connected digraph~$J$ is {\bf elementary} if it is comprised of two strongly connected components~$J_-$ and $J_+$, each of which is a cycle of positive length. Moreover, there exists only one edge from a vertex of $J_-$ to a vertex of $J_+$.    
\end{definition}


\begin{figure}[ht]
	\centering
	\includegraphics[height=50mm]{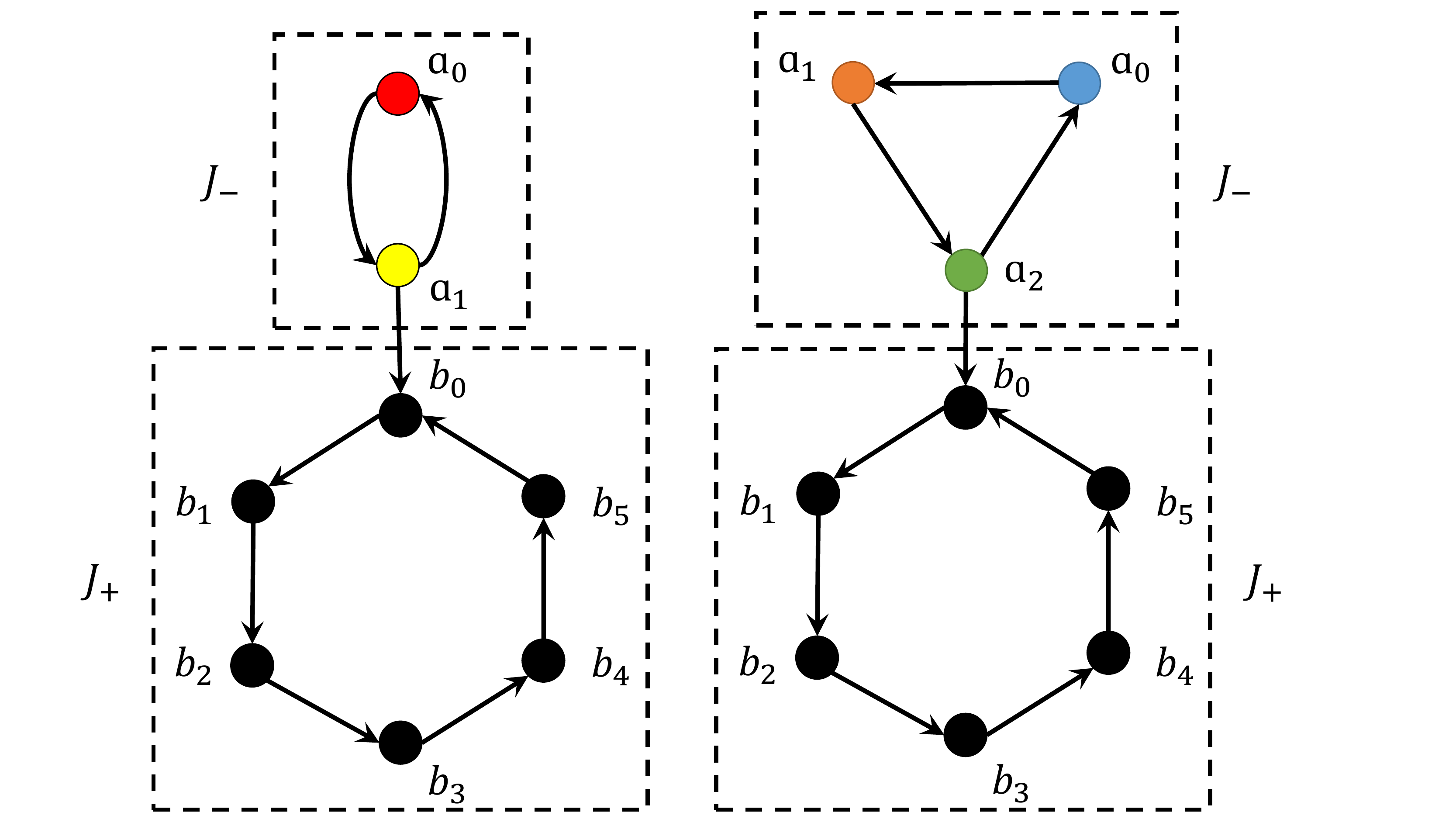}
	\caption{We present here two elementary digraphs. The left (right) digraph has two cycles, with lengths $2$ ($3$) and $6$ ($6$), respectively. The two cycles are connected by a single edge. The vertices have been labeled in a way so that $a_{p_- - 1} b_0$ is the edge from $J_-$ to $J_+$.} 
	\label{elementary}
\end{figure}

We refer the reader to Fig.~\ref{elementary} for an illustration of an elementary digraph. 
From the definition, an elementary digraph~$J$ can be described by a triplet $$J = (J_-, J_+, e),$$ 
where $e$ is the unique edge from $J_-$ to $J_+$. We now fix the elementary digraph~$J$. Let $J_-$ and $J_+$ be comprised of~$p_-$ and~$p_+$ vertices, respectively.
We then write $J_-=(V_-, E_-)$ and $J_+ = (V_+, E_+)$, with 
$$
\left\{
\begin{array}{l}
V_- = \{a_0,\ldots, a_{p_- - 1}\} \\
E_- = \{a_i a_{(i + 1) \bmod p_-} \mid i = 0,\ldots,p_ - - 1\},
\end{array}   
\right.
$$
and similarly, 
$$
\left\{
\begin{array}{l}
V_+ = \{b_0,\ldots, b_{p_+ -1}\} \\
E_+ = \{b_i b_{(i + 1) \bmod p_+} \mid i = 0,\ldots, p_+ -1\}.
\end{array}
\right.
$$
By relabeling the vertices if necessary, we assume that $e = a_{p_- -1}  b_0$ is the edge from $J_-$ to $J_+$. The way we label the vertices of $J$ is rather to facilitate the definition of the map $\omega$ which will be introduced shortly. 

We next consider a CBN whose dependency graph is~$J$. We describe below the asymptotic behavior of the system. Let $x\in \F_2^{|J|}$ be an arbitrary initial state of the system~$J$. With a slight abuse of notation, we let $$x_-:= x_{V_-},\quad and \quad x_+ := x_{V_+},$$ and let $x_{-,i}$ (resp. $x_{+,i}$) be the value of $x$ on vertex $a_i$ (resp. $b_i$).        
We now introduce a map $\omega:\F_2^{|J|} \to \F_2^{|J|}$, which assigns $x$ to a particular state $\omega(x)$ in the periodic orbit which the system~$J$ will enter with $x$ being the initial condition (a precise statement will be given in Lemma~\ref{lem:mapomega} below).  

{\em Definition of $\omega$.} 
Let $\omega_-(x)$ (resp. $\omega_+(x)$) be the restriction of $\omega(x)$ to $V_-$ (resp. $V_+$). 
For $\omega_-(x)$, we simply let $$\omega_-(x):= x_-.$$ 
For $\omega_+(x)$, 
we first let $N_J:= {\rm lcm}\{p_-,p_+\}$, i.e., the least common multiple of $p_-$ and $p_+$. We then set
$$
\omega_{+,i}(x) := x_{+,i} \prod^{N_J/p_+}_{j = 1}  x_{-,(i-jp_+) \bmod p_-},\ \forall i = 0,\ldots, p_+ - 1.
$$
The above definition of $\omega_+$ makes use of the choice we label the vertices of~$J$ in which $e = a_{p_- - 1} b_0$ is the edge from $J_-$ to $J_+$.
We illustrate~$\omega$ in Fig.~\ref{loweromega}.

\begin{figure}[ht]
	\centering
	\vspace{-0.3cm}
	\includegraphics[height=50mm]{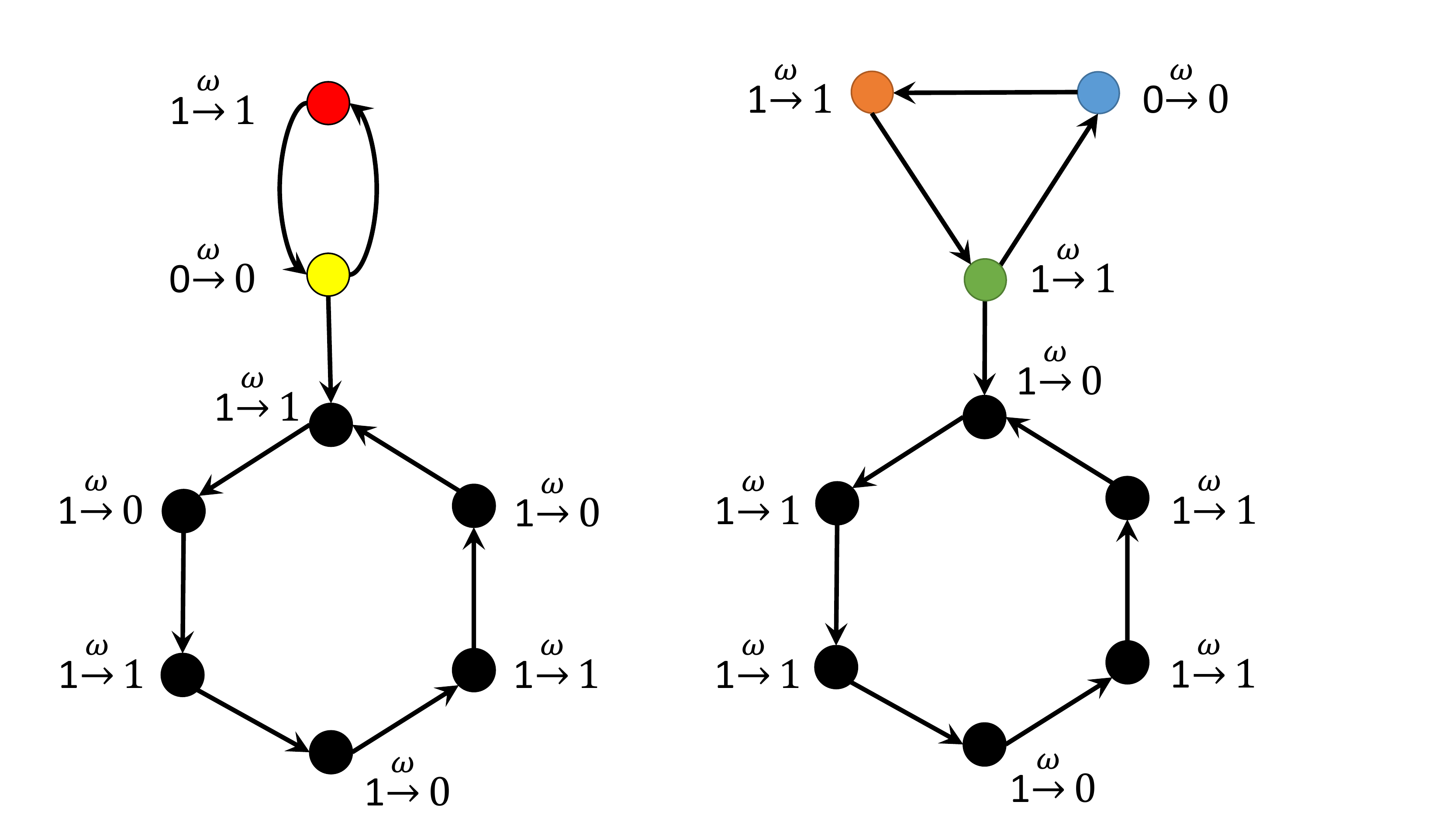}
	\caption{We illustrate the map $\omega$ on the elementary graphs shown in Fig.~\ref{elementary}. The $0/1$'s on the left of the arrows are initial conditions, and the $0/1$'s on the right are the images under the map~$\omega$.}
	\label{loweromega}
\end{figure}

Note that from the definition of $\omega$, we have $\omega(x) \le x$ (entry-wise) for any $x\in \F^{|J|}$. With the map $\omega$ defined above, we have the following result: 

\begin{lemma}\label{lem:mapomega}
	For any initial condition $x(0)$ and any $l > 0$, $x(l N_J) = \omega(x(0))$. 
\end{lemma} 

\begin{proof}
	First, note that $J_-$ is a cycle, and there is no path from a vertex of $J_+$ to a vertex of $J_-$. Thus, for any $l \ge 0$, 
	$x_-(lp_-) = x_-(0) = \omega_-(x(0))$. Since $N_J$ is a multiple of $p_-$, we have $x_-(lN_J) = \omega_-(x(0))$ for any $l \ge 0$. 
	
	We now show that $ x_+(lN_J) = \omega_+(x(0))$ for any $l > 0$.  
	Fix a vertex $b_i$ in $J_+$, we consider the sets 
	${\mathcal N}^{j p_+}_{\rm in}(b_i)$ for $j\ge 1$. 
	From the definition of an elementary digraph and the way the vertices are labeled, we have
	$${\mathcal N}^{j p_+}_{\rm in}(b_i) = \{b_i\} \sqcup \{a_{(i - j'p_+) \bmod p_-} \mid j' = 1,\ldots,j\}.$$ 
	Note that $N_J$ is a multiple of $p_+$. It follows from the above expression that 
	$$
	x_{+,i}(N_J) = \prod_{v_k\in {\mathcal N}^{N_J}_{\rm in}(b_i)} x_k(0) = \omega_{+,i}(x(0)).  
	$$
	Since $N_J$ is also a multiple of $p_-$, for any integer $i$, we have $i \equiv (i- N_J) \bmod p_-$. Thus,   
	${\mathcal N}^{N_J}_{\rm in}(b_i) = {\mathcal N}^{lN_J}_{\rm in}(b_i)$ 
	for any $l > 0$, and hence  
	\begin{align*}
	x_{+,i}(N_J) &= \prod_{v_k\in {\mathcal N}^{N_J}_{\rm in}(b_i)} x_k(0) \\
	&= \prod_{v_k\in {\mathcal N}^{l N_J}_{\rm in}(b_i)} x_k(0) = x_{+,i}(l N_J).
	\end{align*}
	This completes the proof. 
\end{proof}


\subsection{Union of elementary digraphs}\label{ssec:simpleunion}
We consider a CBN $K$ whose dependency graph can be obtained by patching together two elementary digraphs. Specifically, the digraph $K$ satisfies either of the two conditions:  
\begin{enumerate}
	\item[\it i).] The digraph~$K$ is comprised of {\em two} strongly connected components $K_-$ and $K_+$, each of which is a cycle of positive length. Moreover, there exist {\em two} (distinct) edges $e'$ and $e''$ from $K_-$ to $K_+$.
	\item[\it ii).] The digraph~$K$ is comprised of {\em three} strongly connected components $K'_{-}$, $K''_{-}$ and $K_{+}$, each of which is a cycle of positive length. Moreover, there exists {\em an} edge $e'$ from $K'_-$ to $K_+$ and another edge $e''$ from $K''_-$ to $K_+$.
\end{enumerate}
Note that in either case,  the digraph~$K$ can be expressed as a union of two elementary digraphs: In the first case, $K$ is the union of $ J' := (K_-,K_+,e')$ and $J'' := (K_-,K_+,e'')$; in the second case, $K$ is the union of $J' := (K'_-,K_+, e')$ and $J'' := (K''_-,K_+, e'')$. 

We now describe the asymptotic behavior of the CBN~$K$. Similarly, we assign each initial condition $x$ to a particular state in the corresponding periodic orbit. 

First, for a given state $x\in \F^{|K|}_2$, we let $x'_-$, $x''_-$ and $x_+$ be defined by restricting $x$ to $K'_-$, $K''_-$ and $K_+$, respectively. In the case $K$ is of type~{\em i)}, $x'_-$ and $x''_-$ refer to the same state obtained by restricting $x$ to~$K_-$.   
Let $x' \in \F_2^{|J'|}$ and $x''\in \F_2^{|J''|}$ be defined by restricting $x$ to the elementary subgraphs $J'$ and $J''$, respectively. 
We can then apply the map $\omega$ to both $x'$ and $x''$, and obtain two states $\omega(x')$ and $\omega(x'')$. Note, in particular, that both $\omega_+(x')$ and $\omega_+(x'')$ are states of the cycle~$K_+$. 

Next, we introduce an operation~``$\circ$'', which sends the pair $(\omega_+(x'), \omega_+(x''))$ to another state on $K_+$, denoted by $\omega_+(x') \circ \omega_+(x'')$. The definition of $\circ$ is simply the Hadamard product, i.e., entry-wise multiplication. Specifically, given a vertex $b_i$ of the cycle $K_+$, we let $\omega_{+,i}(x')$ (resp. $\omega_{+,i}(x'')$) be the value of $\omega_+(x')$ (resp. $\omega_+(x'')$) on~$b_i$. Then the value of $\omega_+(x') \circ \omega_+(x'')$ on $b_i$ is  given by
$$
\left ( 
\omega_+(x') \circ \omega_+(x'') 
\right )_i := \omega_{+,i}(x') \omega_{+,i}(x'')
$$
We now have the following fact for the CBN K:

\begin{lemma}\label{lem:operationvee}
	Let $p'_-$, $p''_-$ and $p_+$ be the lengths of cycles $J'_-$, $J''_-$ and $J_+$, respectively. Let $N_K:= {\rm lcm}\{p'_-,p''_-,p_+\}$. 
	For any initial condition $x(0)$ and any $l > 0$, 
	\begin{equation}\label{eq:definestateK1}
	\left\{
	\begin{array}{l}
	x'_-(l N_K) = x'_-(0), \\
	x''_-(l N_K) = x''_-(0), \\
	x_+(l N_K) = \omega_+(x'(0)) \circ  \omega_+(x''(0)).
	\end{array}
	\right. 
	\end{equation}\,	
\end{lemma}

\begin{proof}
	For $x'_-(t)$ and $x''_-(t)$, we have 
	$x'_-(l p'_-) = x'_-(0)$ and $x''_-(l p''_-) = x''_-(0)$ 
	for any $l \ge 0$. Since $N_K$ is a multiple of both $p'_-$ and $p''_-$, we have that for any $l \ge 0$, 
	$$
	x'_-(l N_K) = x'_-(0)\quad \mbox{and} \quad x''_-(l N_K) = x''_-(0).
	$$

	We now show that $x_+(l N_K) = \omega_+(x'(0)) \circ  \omega_+(x''(0))$ for any $l > 0$. First, note that for any vertex $b_i$ of $K_+$, we have 
	\begin{equation}\label{eq:decomposeNin}
	{\mathcal N}^{l N_K}_{\rm in}(b_i) = {\mathcal N}^{l N_K}_{\rm in}(b_i; J') \cup {\mathcal N}^{l N_K}_{\rm in}(b_i; J'').
	\end{equation}  
	The two subsets on the right hand side of~\eqref{eq:decomposeNin} are not disjoint; indeed, their intersection is the singleton $\{b_i\}$. 
	Since $x_{+,i}(t)\in \{0,1\}$, $x_{+,i}^2(t) = x_{+,i}(t)$. By~\eqref{eq:decomposeNin}, we factorize $x_{+,i}(l N_K)$ as follows:  
	\begin{align*}
	x_{+,i}(l N_K) &= \prod_{v_k\in {\mathcal N}^{l N_K}_{\rm in}(b_i)} x_k(0) \\
	&=  
	\prod_{v_{k'}\in {\mathcal N}^{l N_K}_{\rm in}(b_i; J')} x'_{k'}(0) \prod_{v_{k''}\in {\mathcal N}^{l N_K}_{\rm in}(b_i; J'')} x''_{k''}(0). 
	\end{align*}
	Since $N_K$ is a multiple of $p'_-$, $p''_-$, and $p_+$, we apply the same argument as in the proof of Lemma~\ref{lem:mapomega} and identify the two factors in the above expression as follows: 
	$$
	\left\{
	\begin{array}{l}
	\prod_{v_{k'}\in {\mathcal N}^{l N_K}_{\rm in}(b_i; J')} x'_{k'}(0) = \omega_{+,i}(x'(0)), \\
	\prod_{v_{k''}\in {\mathcal N}^{l N_K}_{\rm in}(b_i; J'')} x''_{k''}(0) = \omega_{+,i}(x''(0)),
	\end{array}
	\right.
	$$
	which holds for any $l > 0$. We thus obtain $$x_{+,i}(l N_k) = \omega_{+,i}(x'(0))\omega_{+,i}(x''(0)).$$ 
	This completes the proof.  
\end{proof}

{\em A mild generalization.} 
We now consider any reduced weakly connected digraph $H$ that satisfies the following condition: Let ${\mathcal S}(H)$ be the collection of strongly connected components of $H$, which are all cycles of positive length. Further, decompose ${\mathcal S}(H) = \sqcup^L_{l = 0}{\mathcal S}^l(H)$, with ${\mathcal S}^{l}(H)$ the immediate successor of ${\mathcal S}^{l-1}(H)$. Then, $L = 1$, and moreover, ${\mathcal S}^1(H)$ is comprised of only one cycle. 

It should be clear that any elementary digraph $J$, or the digraph $K$, satisfies the condition above. It is also clear that any such digraph $H$ can be obtained by taking the union of several elementary digraphs which share the same ``lower'' cycle. We express the digraph $H$ as $H = \cup^m_{k = 1} J^{(k)}$,
where each $J^{(k)} = (H^{(k)}_-, H_+, e^{(k)})$ is an elementary digraph. 
An illustration of $H$ is in Fig.~\ref{mild}. 

\begin{figure}[ht]
	\centering
	\includegraphics[height=49mm]{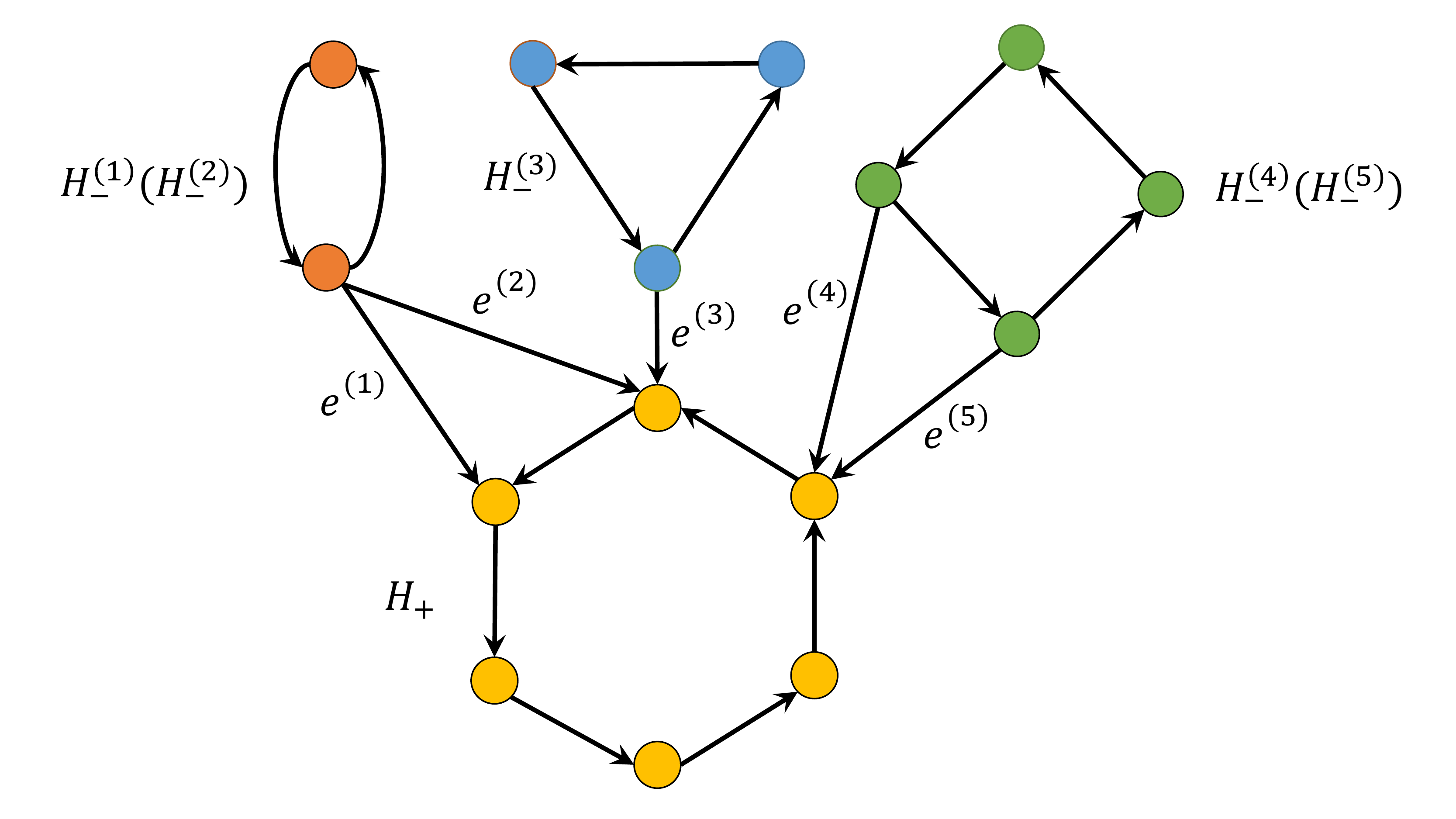}
	\caption{We provide here a reduced digraph $H$ obtained by patching $5$ elementary digraphs. Specifically, we have $H=\cup^5_{k = 1} J^{(k)}$, with $J^{(k)} = (H^{(k)}_-, H_+, e^{(k)})$ illustrated in the figure. Note, in particular, that $H^{(1)}_-=H^{(2)}_-$ and $H^{(4)}_-=H^{(5)}_-$.}
	\label{mild}
\end{figure}

Let $x(t)$ be a state of the CBN $H$, and $x^{(k)}(t)$ be the restriction of $x(t)$ to $J^{(k)}$. 
Since the Hadamard product is associative and commutative, there is no ambiguity to define 
$$
\bigcirc^m_{k = 1}\omega_+(x^{(k)}) := \omega_+(x^{(1)})\circ \cdots \circ \omega_+(x^{(m)}). 
$$
The following fact then generalizes Lemma~\ref{lem:operationvee}:

\begin{lemma}\label{lem:elemgeneral}
	Let $p^{(k)}_-$ be the length of cycle $H^{(k)}_-$, and $p_+$ be the length of cycle $H_+$. Let $N_{H}: = {\rm lcm}\{p_+, p^{(1)}_-,\ldots, p^{(m)}_-\}$. For any initial condition $x(0)$ of the CBN $H$ and any $l > 0$, 
	\begin{equation}\label{eq:definestateK}
	\left\{
	\begin{array}{l}
	x^{(k)}_-(l N_H) = x^{(k)}_-(0),\quad \forall k = 1,\ldots, m, \\
	x_+(l N_H) = \bigcirc^m_{k = 1}\omega_+(x^{(k)}(0)).
	\end{array}
	\right. 
	\end{equation}\,  
\end{lemma}

We omit the proof as it is similar to the proof of Lemma~\ref{lem:operationvee}. 


\subsection{On a general reduced system}\label{ssec:general}
We now extend the results established in the previous subsections to the case where $H$ is an arbitrary reduced digraph. But for simplicity of exposition, we assume that all cycles (i.e. the strongly connected components) of $H$ have positive lengths. The analysis for the most general case follows similarly. A discussion will be included at the end of the section.   

Our objective here is to generalize the map~$\omega$ to a map~$\Omega$, which sends an initial condition~$x$ of the  system~$H$ to a state $\Omega(x)$ in the corresponding periodic orbit. 
In particular,  
we show below that the map $\Omega$ can be obtained by repeatedly applying the map~$\omega$ and the Hadamard product~``$\circ$''. 

{\em Definition of $\Omega$.} 
Partition $\cS(H)$ into $\cS(H) = \sqcup^L_{l = 0} \cS^l(H)$.  
For a given state~$x$, we define $\Omega(x)$ by specifying its value on each strongly connected component $H_i$, denoted by $\Omega_i(x)$. The definition will be carried out by induction on the number~$l$:

{\em Base case $l = 0$.} For each $H_i = (U_i, F_i)\in \cS^0(H)$, we define 
$
\Omega_i(x) := x_{U_i} 
$, i.e., $\Omega$ is the identity map when restricted to each~$H_i\in \cS^0(H)$.

{\em Inductive step.} We now assume that $\Omega_j(x)$ has been defined for every  $H_j\in \sqcup^{l-1}_{k = 0}\cS^k(H)$, for some $l\ge 1$. We now define $\Omega_i(x)$ for each $H_i\in \cS^{l}(H)$. Since~$H$ is connected, for a given $H_i\in \cS^{l}(H)$, there exists at least one $H_j\in \cS^{k}(H)$, for $k < l$, and an edge~$e$ from~$H_j$ to~$H_i$. Thus, there exists at least an elementary subgraph~$J = (J_-, J_+, e)$ of $H$ such that $J_- = H_j$, $J_+ = H_i$, and~$e$ is the unique edge from $J_-$ to $J_+$.  
We now let 
${\rm Elem}(H_i)$ 
be the collection of any such elementary subgraphs of $H$, i.e., if $J\in {\rm Elem}(H_i)$, then $J_+ = H_i$.    

Now, for a given~$J\in {\rm Elem}(H_i)$, we define a state $x^{J}$ on~$J$ as follows: Let $x^{J}_-$ (resp. $x^{J}_+$) be the restriction of $x^J$ on $J_-$ (resp. $J_+$). 
Suppose that $J_- = H_j$ for some $H_j \in \cS^{k}(H)$ with $k < l$; then, $\Omega_j(x)$ is defined by the induction hypothesis, and we set 
$$
x^J_- := \Omega_j(x) \hspace{10pt} \mbox{and} \hspace{10pt} x^J_+ := x_{U_i}. 
$$   
With the $x^{J}$'s as defined above, we then set 
\begin{eqnarray}\label{Omega}
\Omega_i(x) := \bigcirc_{J\in {\rm Elem}(H_i)}\omega_+(x^{J}).  
\end{eqnarray}
We have thus defined the map~$\Omega$. Note that for any $x\in \F^{|H|}$, we have $\Omega(x) \le x$. We further illustrate $\Omega$ via the following example. 

{\color{black}
	{\em Example.}
	We consider the reduced system $H$ in Fig.~\ref{omega}, with the initial conditions of the vertices given on the left of the arrows. There are four strongly connected components (cycles), 
	labeled as $H_i = (U_i, F_i)$ for $i = 1,\ldots, 4$.  
	It should be clear that $\cS^0(H)=\{H_1,H_2\}$, $\cS^1(H)=\{H_3\}$, and $\cS^2(H)=\{H_4\}.$ 
	We now illustrate the map $\Omega$. First, note that $\Omega_1$ and $\Omega_2$ are the identity maps when restricted to $U_1$ and $U_2$. Next, for $\Omega_3$, we have that ${\rm Elem}(H_3)=\{J_1,J_2\}$ where $J_1:=(H_1,H_3,e_1)$ and $J_2:=(H_2,H_3,e_2)$, and hence 
	$$
	\Omega_3(x) = \bigcirc_{J\in \{J_1,J_2\}}\omega_+(x^{J})=\omega_{+_1}(x^{J_1})\circ\omega_{+_2}(x^{J_2}).  
	$$
	Finally, for $\Omega_4$, we have that ${\rm Elem}(H_4)=\{J_3,J_4\}$ where $J_3:=(H_1,H_4,e_3)$ and $J_4:=(H_3,H_4,e_4)$, and hence 
	$$
	\Omega_4(x) = \bigcirc_{J\in \{J_3,J_4\}}\omega_+(x^{J})=\omega_{+_3}(x^{J_3})\circ\omega_{+_4}(x^{J_4}).  
	$$
	The output values of the $\Omega_i$'s are illustrated in Fig.~\ref{omega}. 
}

\begin{figure}[ht]
	\centering
	\includegraphics[height=80mm]{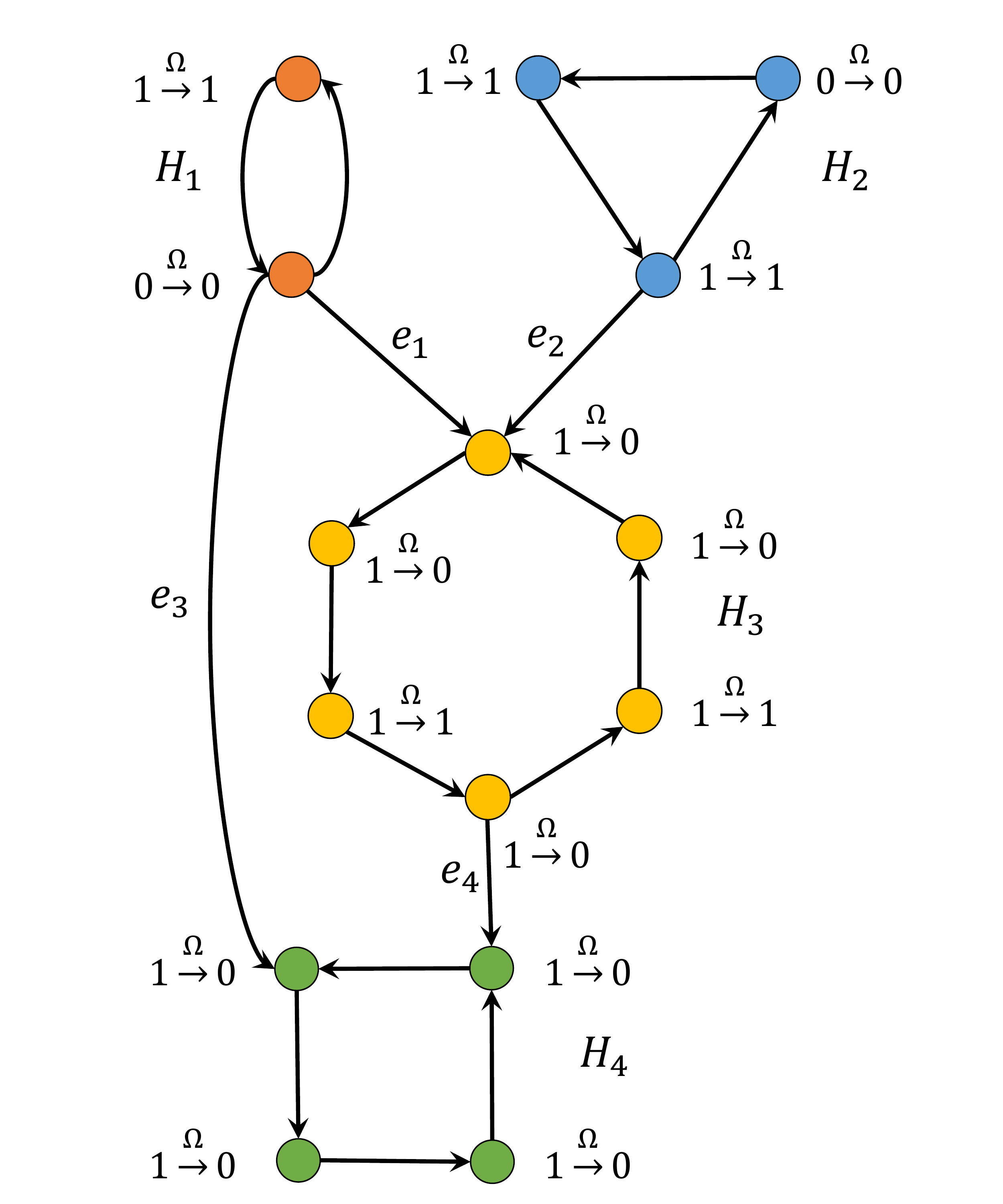}
	\caption{We illustrate here the map $\Omega$. The $0/1$'s on the left of the arrows are initial conditions, and the $0/1$'s on the right are the images under the map $\Omega$.}
	\label{omega}
\end{figure}

With the map $\Omega$ at hand, we state the second main result of the paper: 

\begin{theorem}\label{maintheorem}
	Let $x(0)$ be an initial condition of a reduced system~$H$. 
	There exists a time step $N$, as a multiple of lengths of the cycles in $H$,  such that $x(l N) = \Omega(x(0))$ for any $l > 0$. 
\end{theorem} 

\begin{Remark}
	Theorem~\ref{maintheorem}, combined with the second item of Theorem~\ref{pro:dynamics}, implies that from $t = N$, the dynamics of system $H$ proceeds as if $H$ was comprised of disjoint cycles $H_i$'s. Let $p_i$ be the length of $H_i$. Then, 
	$x_{U_i}(N + lp_i) = \Omega_i(x(0))$ for any $l \ge 0$. 
\end{Remark}

\subsection{Proof of Theorem~\ref{maintheorem}}\label{ssec:proofthm2}
We have $\cS(H) = \sqcup^L_{l = 0} \cS^l(H)$. 
The proof will be carried out by induction on the number~$L$. 

{\em Base case: $L = 1$.} 
In this case, the time step $N$ can be taken as the least common multiple of all cycle lengths. We show below that $x(l N) = \Omega(x)$ for any $l > 0$. First, note that if a cycle $H_i = (U_i,F_i)$ of length $p_i$ belongs to $\cS^0(H)$, then $x_{U_i}(l p_i) = x_{U_i}(0) = \Omega_i(x)$ for any $l \ge 0$. Since $p_i$ divides $N$, we have $x_{U_i}(l N) = \Omega_i(x)$. 

Next, we pick a cycle $H_i$ out of $\cS^1(H)$. Recall that ${\rm Elem}(H_i)$ is the collection of any elementary subgraph $J$ of $H$ such that $J_+ = H_i$. 
For convenience, we let $H^\star_i$ be the subgraph of $H$ obtained by taking the union of these elementary subgraphs, i.e., $H^\star_i = \cup_{J\in {\rm Elem}(H_i)} J$.  
It should be clear that the state $x_{U_i}(t)$ of $H_i$ depends only on the initial condition of $H^\star_i$, but not of any other component of $H$. Thus, we can appeal to Lemma~\ref{lem:elemgeneral} and obtain that for any $l> 0$,
$$
x_{U_i}(l N_{H^\star_{i}}) =  \bigcirc_{J\in {\rm Elem}(H_i)}\omega_+(x(0)^{J}) = \Omega_i(x(0)),
$$
where $N_{H^\star_{i}}$ is the least common multiple of the lengths of cycles in the subgraph $H^\star_i$. Since $N_{H^\star_{i}}$ divides $N$, we have $x_{U_i}(l N) = \Omega_i(x(0))$ for all $l > 0$. This holds for all $H_i\in \cS^1(H)$, and hence $x(l N) = \Omega(x(0))$ for all $l > 0$.  

{\em Inductive step.} We assume that Theorem~\ref{maintheorem} holds for $L -1$ for $L > 1$, and prove for $L$. We first let 
$U':= U - \sqcup_{H_i\in \cS^{L}(H)} U_i $, i.e., $U'_i$ is the collection of vertices which do not belong to any $H_i$ in $\cS^{L}(H)$. We then let $H'$ be the subgraph of $H$ induced by $U'$. Note that the state $x_{U'}(t)$ of $H'$ depends only on the initial condition of $H'$, but not of any $H_i\in \cS^{L}(H)$. In other words, if only $x_{U'}(t)$ needs to be determined, then we can simply investigate the system $H'$ 
with $x_{U'}(0)$ the initial condition.

The digraph $H'$ may be comprised of multiple (weakly) connected components. We label them as $H'^{(1)},\ldots,H'^{(m)}$. For each connected component~$H'^{(k)}$, we decompose $\cS(H'^{(k)}) = \sqcup^{L^{(k)}}_{l =1}\cS(H'^{(k)})$. It should be clear that $L^{(k)} < L$ for all $k =1,\ldots, m$. Thus, the induction hypothesis applies to every $H'^{(k)}$, and hence to their (disjoint) union~$H'$.  
In particular, there exists a time step $N'$, as a multiple of lengths of the cycles in $H'$, such that $x_{U'}(l N') = \Omega_{H'}(x_{U'}(0))$  for any $l>0$, where the subindex $H'$ in $\Omega_{H'}$ indicates that the map $\Omega$ is applied to the digraph $H'$. 
Note that if we let $N''$ be a multiple of $N'$, then it still holds that $x_{U'}(l N'') = \Omega_{H'}(x_{U'}(0))$ for all $l > 0$. Thus, we increase $N'$ if necessary so that $N'$ is a multiple of lengths of all cycles in~$H$.

We now consider the state $x_{U_i}(t)$ for $H_i \in \cS^{L}(H)$. 
To proceed, we first introduce two new initial conditions $x'(0)$ and $x''(0)$ for the system $H$: The first one $x'(0)$ is given by 
$$
\left\{
\begin{array}{l}
x'_{U'}(0) := \Omega_{H'}(x_{U'}(0)), \\ 
x'_{U_i}(0) := x_{U_i}(0), \quad \forall H_i\in \cS^{L}(H).
\end{array}
\right.
$$
The other one $x''(0)$ is given by
$$
\left\{
\begin{array}{l}
x''_{U'}(0) := \Omega_{H'}(x_{U'}(0)), \\ 
x''_{U_i}(0) := x_{U_i}(N'), \quad \forall H_i\in \cS^{L}(H).
\end{array}
\right.
$$

We state below a few facts about $x'(0)$ and $x''(0)$. First, for $x'(0)$, note that $\Omega_{H'}(x_{U'}(0))\le x_{U'}(0)$, and hence $x'(0) \le x(0)$. Next, for $x''(0)$, since $N'$ is a multiple of $p_i$, we have $x_{U_i}(N') \le x_{U_i}(0)$ for any $H_i\in \cS^{L}(H)$, and hence $x''(0) \le x'(0)$. We thus have $x''(0) \le x'(0) \le x(0)$, and hence for any $l  \ge 0$, 
\begin{equation}\label{eq:pfthm2}
x''(l N') \le x'(l N') \le x(l N').
\end{equation}
On the other hand, by construction, we have $x(N') = x''(0)$, and hence for any $l \ge 0$, 
\begin{equation}\label{eq:pfthm3}
x((l +1 )N') = x''(l N').
\end{equation}
Combining~\eqref{eq:pfthm2} and~\eqref{eq:pfthm3}, we obtain for $l \ge 0$, 
\begin{equation}\label{eq:pfthm4}
x'((l+1)N') \le x((l+1) N') = x''(l N')\le x'(l N').
\end{equation}

We show below that 
\begin{equation}\label{eq:pfthm5}
x'(l N') = \Omega(x(0)), \quad \forall l > 0.
\end{equation}
It suffices to show that for any $H_i\in \cS^{L}(H)$, we have 
$
x'_{U_i}(l N') = \Omega_i(x(0))
$ for any $l > 0$. 
Fix an $H_i\in \cS^{L}(H)$, and similarly, let $H^\star_i:= \cup_{J\in {\rm Elem}(H_i)} J$. Since $x'_{U'}(0)$ is already in a periodic orbit of system $H'$, we know from Theorem~\ref{pro:dynamics} that the dynamics of $x'_{U'}(t)$ proceeds as if $H'$ was comprised of disjoint cycles. 
This, in particular, implies that the state $x'_{U_i}(t)$ of $H_i$ depends only on the initial condition of $H^\star_i$, but not of any other component of $H$. We thus appeal again to Lemma~\ref{lem:elemgeneral} and obtain that for any $l> 0$,
$$
x'_{U_i}(l N_{H^\star_{i}}) =  \bigcirc_{J\in {\rm Elem}(H_i)}\omega_+(x(0)^{J}) = \Omega_i(x(0)),
$$
where $N_{H^\star_{i}}$ is the least common multiple of the lengths of cycles in the subgraph $H^\star_i$. Since $N_{H^\star_{i}}$ divides $N'$, we have $x'_{U_i}(l N') = \Omega_i(x(0))$ for all $l > 0$. We have thus shown that~\eqref{eq:pfthm5} holds.  

Now, from~\eqref{eq:pfthm4} and~\eqref{eq:pfthm5}, we have
$x(l N') = \Omega(x(0))$ for all $l \ge 2$.
Thus, if we let $N := 2N'$, then $x(l N) = \Omega(x(0))$ for all $l > 0$. This completes the proof. 
\hfill\qed

\begin{Remark}
	Taking a closer look at the proof of Theorem~\ref{maintheorem}, we have the following fact about the transition time for the system $H$ to enter a periodic orbit from an arbitrary initial condition. Let $N_H$ be the least common multiple of lengths of the cycles in $H$. Then, for any initial condition $x(0)$ of system $H$, $x(2^{L - 1}N_H)$ is in a periodic orbit. 
\end{Remark}

We conclude this subsection with a discussion on the case where the reduced digraph $H$ has a cycle of length~$0$, i.e., a single vertex {\em without} self-arc. The analysis for the asymptotic behavior for such a case does not differ too much from the analysis we had earlier. We sketch below a few key modifications one needs to accommodate into the existing arguments. 

First, instead of using $\cS(H)$, we define $\cS_+(H)$ as the collection of cycles of {\em positive lengths} in $H$. Since $\cS_+(H)$ is a subset of $\cS(H)$, the partial order ``$\succ$'' defined on $\cS(H)$ induces a partial order on $\cS_+(H)$. Similarly, we decompose $\cS_+(H) = \sqcup^L_{l = 0}\cS^l(H)$ so that $\cS^l(H)$ is an immediate successor of $\cS^{l-1}(H)$.  

Next, we modify the definition of an elementary digraph $J$ so that it is now comprised of an ``upper'' cycle $J_-$, a ``lower'' cycle $J_+$, and a {\em path} (instead of an edge) from $J_-$ to $J_+$. Note that with the modified version, we can still ``patch'' multiple elementary digraphs if their ``down'' cycles share the same length.   

Correspondingly, the map $\omega$ also needs to be rectified. 
Let the connecting path of an elementary graph $J$ be comprised of more than one edge (otherwise, $J$ agrees with our earlier definition). We denote it by $J_0 = a_{p_- -1} c_0\cdots c_{m-1} b_0$ for  $m \ge 1$. Let $\omega_-, \omega_+$ and $\omega_0$ be defined by restricting $\omega$ to $V_-$, $V_+$ and the vertices $c_i$'s, respectively. We still let $\omega_-$ be the identity map. For $\omega_0$, we let $\omega_{0,i}(x) := x_{-,i \bmod p_-}$ for any $i = 0,\ldots, m-1$. For $\omega_+$, we let 
\begin{align*}
&\omega_{+,i}(x) :=\\
& x_{+,i} \prod^{\lfloor \frac{i + m}{p_+}\rfloor}_{j = 1}x_{0,(i - jp_+) \bmod m} 
\prod^{\lfloor \frac{i + m}{p_+}\rfloor + \frac{N_J}{p_+}}_{ j = \lfloor \frac{i + m}{p_+}\rfloor + 1} x_{-,(i - jp_+) \bmod p_-},
\end{align*}
for any $i = 0,\ldots, p_+-1$, where $\lfloor \cdot \rfloor$ is the floor function. Lemmas~\ref{lem:mapomega},~\ref{lem:operationvee}, and~\ref{lem:elemgeneral} still hold. Yet, the time steps $N_J$, $N_K$ and $N_H$ in the statements are  multiples of lengths of the cycles in $J$, $K$ and $H$, respectively, but may not be the least common multiples.  
We provide two examples of the modified map $\omega$ in Fig.~\ref{general}.

\begin{figure}[ht]
	\centering
	\includegraphics[height=50mm]{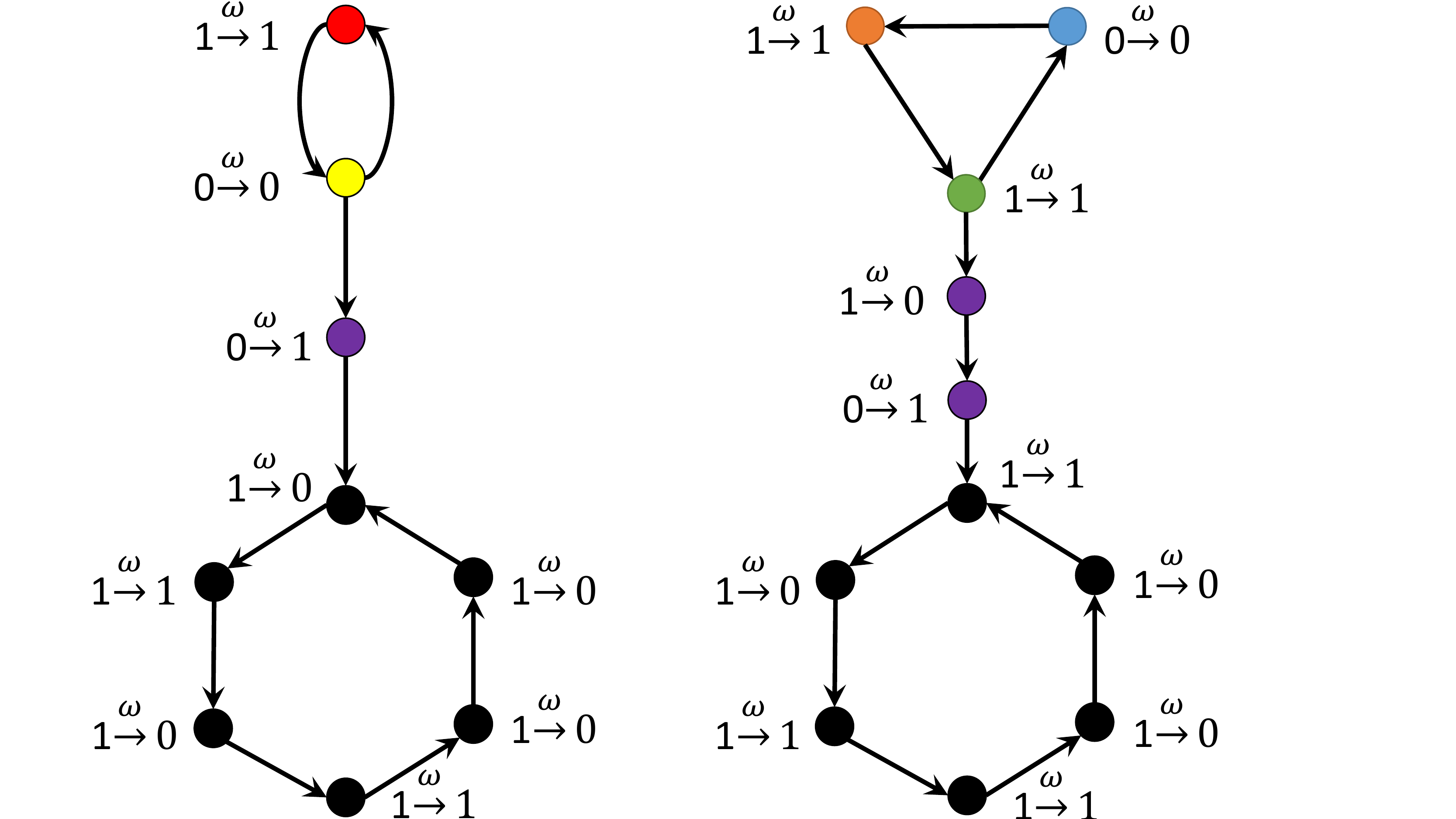}
	\caption{We illustrate here the modified map $\omega$. The $0/1$'s on the left of the arrows are initial conditions, and the $0/1$'s on the right are the images under the map $\omega$.}
	\label{general}
\end{figure}

Finally, for the map $\Omega$, we modify the definition by replacing $\cS(H)$ with $\cS_+(H)$ and using the rectified elementary graphs and the map $\omega$. Further, we note that the co-domain of $\Omega$ is not the entire state of $H$, but rather its restriction to cycles of positive lengths. On the other hand, if a state~$x(t)$ is in a periodic orbit of system $H$, then $x(t)$ can be uniquely determined by the $x_{U_i}(t)$'s where the $H_i$'s are cycles in $H$ of positive lengths. Thus, the map $\Omega$ completely characterizes the asymptotic behavior of system $H$.

\section{Conclusions and Outlooks}\label{end}
We have characterized, in this paper, the asymptotic behavior of CBNs over weakly connected digraphs. In particular, we have provided a complete answer to the question of which periodic orbit the system will enter with a given initial condition.  
Along the analysis, we have introduced a new graphical-approach, termed {\em system reduction}, for studying this type of problem. We have shown that the reduced system uniquely determines the asymptotic behavior of the original system. Such an approach significantly simplifies the analysis, and moreover, it could be modified and applied to other types of Boolean networks.  

There are several research directions we will pursue in the future as an outgrowth of this study. First, we recall from Corollary~\ref{lem3:thm1} that the period of any periodic orbit of a CBN $G$ divides $N_G$, the least common multiple of lengths of the cycles in $G$. Yet, we note that it is not true that any divisor of $N_G$ can be the period of a periodic orbit. We can thus ask the question of: {\em what are the possible periods of periodic orbits of the CBN $G$?} We can further ask: {\em what is the relationship between the network topology $G$ and the possible periods?}

Second, we can address issues about the stability structure of the periodic orbits of a CBN. Consider the situation where $x(0)$ is in a periodic orbit of a CBN G, and yet at a certain time step~$t$, there is a vertex, say $v_i$, whose value does not follow the update rule, i.e., $x_i(t) = \bar f_i(x(t - 1))$. Then, with the current state $x(t)$ as an ``initial condition'', the system will enter another (possibly the same) periodic orbit. The stability structure of a CBN $G$ is then about characterizing all such transitions from one periodic orbit to another. This problem has been investigated in~\cite{stabilityfull}, but only for the case where $G$ is strongly connected.

Third, we are also interested in orbit/state-controllability of weakly connected CBNs: Assuming that there is a selected subset of variables whose values are determined by external inputs, we ask {\em whether one is able to steer the system to any desired periodic orbit or state, by controlling the sequence of values of the external inputs.} For steering the system to any desired state, the question has been addressed in~\cite{gaocontrollability} and~\cite{weiss2017minimal} most recently. For steering the system to any desired periodic orbit,~\cite{gaocontrollability} has addressed the case when the underlying graph is strongly connected. We believe that the results developed in this paper will be be useful in investigating the orbit-controllability of general weakly connected CBNs.

\bibliographystyle{unsrt}
\bibliography{references}

%

\begin{IEEEbiography} [{\includegraphics[width=1in,height=1.25in,clip,keepaspectratio]{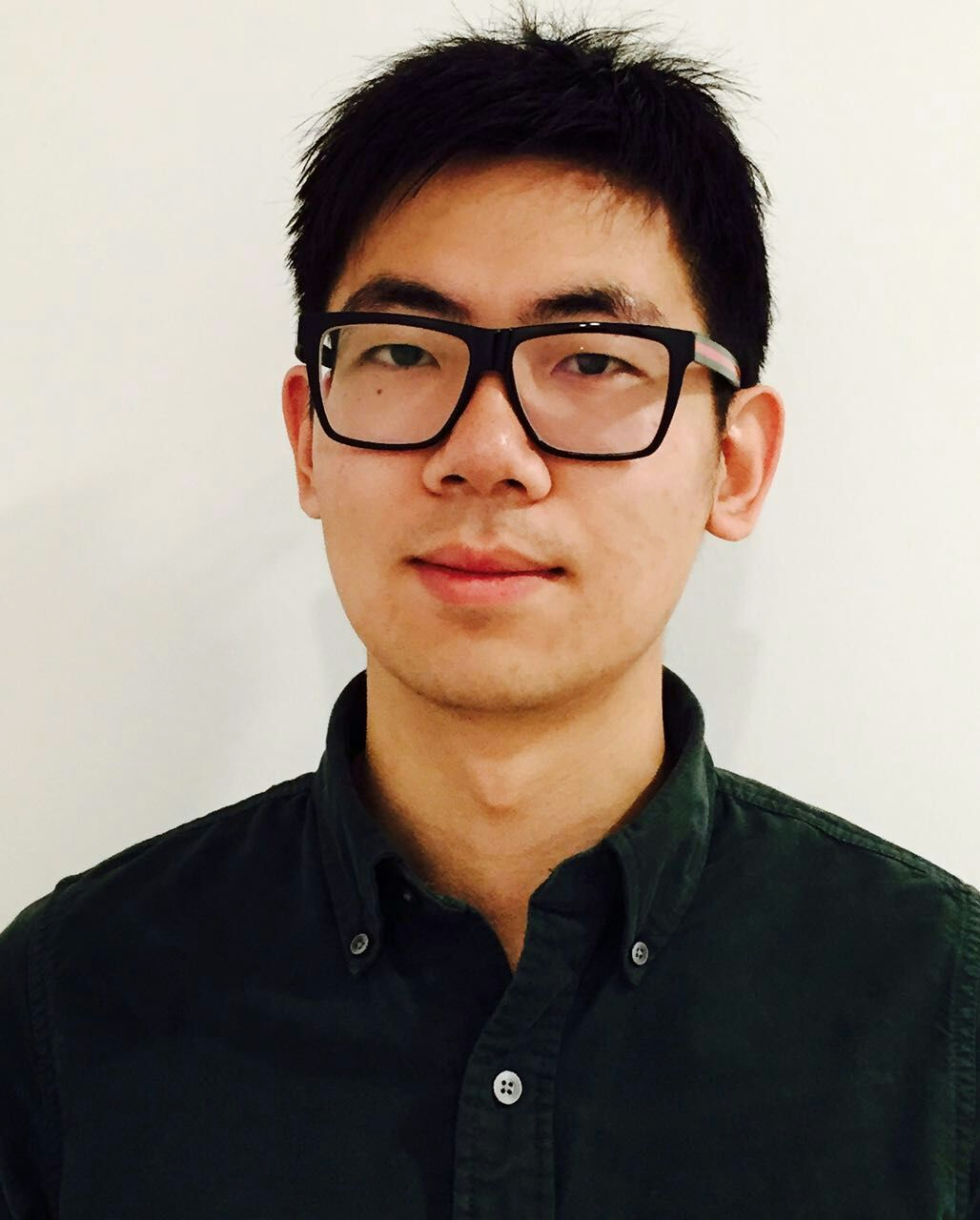}}]
	{Xudong Chen} is now Assistant Professor in the Department of Electrical, Computer and Energy Engineering at the University of Colorado, Boulder. Before that, he was a postdoctoral fellow in the Coordinated Science Laboratory at the University of Illinois, Urbana-Champaign. He obtained the B.S. degree from Tsinghua University, Beijing, China, in 2009, and the Ph.D. degree in Electrical Engineering from Harvard University, Cambridge, Massachusetts, in 2014. His research interests are in the area of control theory, stochastic processes, optimization, game theory and their applications in modeling, control, and estimation of networked systems and ensemble systems. 
\end{IEEEbiography}

\begin{IEEEbiography} [{\includegraphics[width=1in,height=1.25in,clip,keepaspectratio]{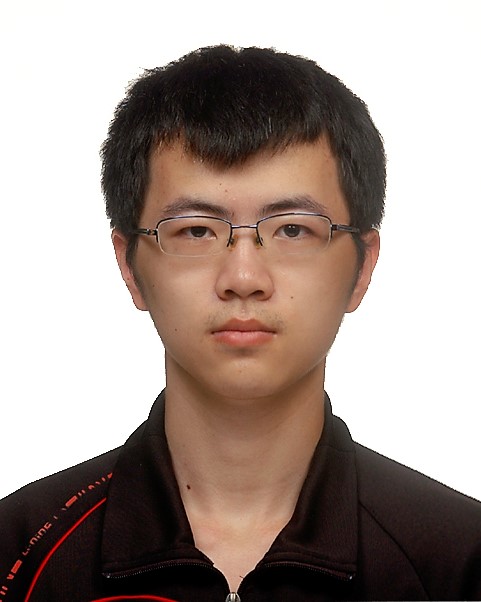}}]
	{Zuguang Gao} received the M.S. degree (in 2017) and the B.S. degree (in 2015), both in Electrical Engineering, from the University of Illinois at Urbana-Champaign. He is currently pursuing Ph.D. in Management Science/Operations Management at the University of Chicago Booth School of Business. His research interests include queueing and scheduling theory, algorithm design and analysis, game theory, control and modeling of complex systems, large scale networked dynamical systems and their applications.
\end{IEEEbiography}

\begin{IEEEbiography} [{\includegraphics[width=1in,height=1.25in,clip,keepaspectratio]{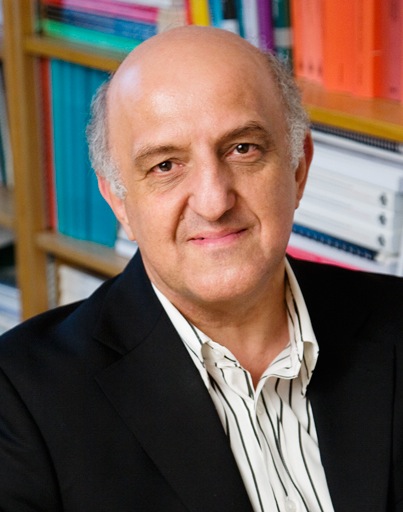}}]
	{\bf Tamer Ba\c{s}ar} (S'71-M'73-SM'79-F'83-LF'13) is with the University of Illinois at Urbana-Champaign, where he holds the academic positions of  Swanlund Endowed Chair;    
Center for Advanced Study Professor of  Electrical and Computer Engineering; 
Research Professor at the Coordinated Science
Laboratory; and Research Professor  at the Information Trust Institute. 
He is also the Director of the Center for Advanced Study.
He received B.S.E.E. from Robert College, Istanbul,
and M.S., M.Phil, and Ph.D. from Yale University. He is a member of the US National Academy
of Engineering,  member of the  European Academy of Sciences, and Fellow of IEEE, IFAC (International Federation of Automatic Control) and SIAM (Society for Industrial and Applied Mathematics), and has served as president of IEEE CSS (Control Systems  Society), ISDG (International Society of Dynamic Games), and AACC (American Automatic Control Council). He has received several awards and recognitions over the years, including the
highest awards of IEEE CSS, IFAC, AACC, and ISDG, the IEEE Control Systems Award, and a number of international honorary doctorates and professorships. He has over 900 publications in systems, control, communications, networks,
and dynamic games, including books on non-cooperative dynamic game theory, robust control,
network security, wireless and communication networks, and stochastic networked control. He was
the Editor-in-Chief of Automatica between 2004 and 2014, and is currently  editor of several book series. His current research interests include stochastic teams, games, and networks; security; distributed computation and learning; and cyber-physical systems.
\end{IEEEbiography}









\end{document}